\documentclass[11pt]{amsart}

\usepackage{amsmath,amsfonts,amsthm,amssymb}
\usepackage{url,hyperref}

\usepackage{amsrefs}

  \def\setupbib{\catcode`@=\active}
\begingroup\lccode`~=`@
  \lowercase{\endgroup\def~}#1#{\gatherkey{#1}}
\def\gatherkey#1#2{\gatherkeyaux{#1}#2\gatherkeyaux}
\def\gatherkeyaux#1#2,#3\gatherkeyaux{\bib{#2}{#1}{#3}}  

\usepackage[toc,page]{appendix}

\usepackage{longtable,tabularx,tabulary}
\usepackage{float,mathtools}
\usepackage{color}
\usepackage{anysize,placeins,enumitem}
\usepackage{latexsym}
\usepackage{graphicx}
\usepackage{xcolor}
\usepackage{textgreek}
\usepackage{textcomp}
\usepackage{cases}
\usepackage{url}
\usepackage{algpseudocode, comment}
\usepackage{todonotes}
\usepackage{mathrsfs}
\usepackage{caption}

 \DeclareMathOperator{\N}{N}

\newcommand{\R}{{\mathbb{R}}}

\newcommand{\Q}{{\mathbb{Q}}}

\renewcommand{\Re}{\operatorname{Re}}
\renewcommand{\Im}{\operatorname{Im}}

\renewcommand{\b}{\beta}
\newcommand{\g}{\gamma}

\newtheorem{theorem}{Theorem}[section]
 \newtheorem{corollary}[theorem]{Corollary}
 \newtheorem{lemma}[theorem]{Lemma}
 
 \newtheorem{proposition}[theorem]{Proposition}

\theoremstyle{remark}
\newtheorem*{remark}{Remark}

\usepackage{comment}

\numberwithin{equation}{section}
\reversemarginpar 

\begin{document}

\title[Zero-free regions for Dedekind zeta-functions]{New zero-free regions for Dedekind zeta-functions\\at small and large ordinates}

\author[S.~Das]{Sourabhashis Das}
\address[Sourabhashis Das]{Department of Pure Mathematics\\
University of Waterloo\\
Ontario, Canada}
\email{s57das@uwaterloo.ca}
\urladdr{\url{https://souravasis.wixsite.com/sourabhdas}}

\author[S.~Gaba]{Swati Gaba}
\address[Swati Gaba]{Department of Mathematics\\
University of Connecticut, US}
\email{swati.gaba@uconn.edu}
\urladdr{\url{https://math.uconn.edu/person/swati-gaba/}}

\author[E.~S.~Lee]{Ethan Simpson Lee}
\address[Ethan Simpson Lee]{University of Bristol, School of Mathematics, Fry Building, Woodland Road, Bristol, BS8 1UG, UK} 
\email{ethan.lee@bristol.ac.uk}
\urladdr{\url{https://sites.google.com/view/ethansleemath/home}}

\author[A.~Savalia]{Aditi Savalia}
\address[Aditi Savalia]{Department of Mathematics\\
Indian Institute of Technology Bombay\\
Mumbai, India}
\email{aditis@math.iitb.ac.in, aditisav1995@gmail.com}
\urladdr{\url{https://sites.google.com/view/aditisavalia/home}}

\author[P.-J.~Wong]{Peng-Jie Wong}
\address[Peng-Jie Wong]{Department of Applied Mathematics\\
National Sun Yat-Sen University\\
Kaohsiung City, Taiwan}
\email{pjwong@math.nsysu.edu.tw}
\urladdr{\url{https://www-math.nsysu.edu.tw/~pjwong/}}

\thanks{This research was partially supported by the NSTC grant 111-2115-M-110-005-MY3 of PJW. ESL thanks the Heilbronn Institute for Mathematical Research for their support. SD was supported by a University of Waterloo graduate fellowship. SG thanks the Department of Mathematics at the University of Connecticut for partial funding. All of the authors thank the organisers of the IPENT summer school, who made this project possible.}

\begin{abstract}
Given a number field $L\neq \mathbb{Q}$, 
we obtain new and explicit zero-free regions for Dedekind zeta-functions of $L$, which refine the previous works of Ahn--Kwon, Kadiri, and Lee. In particular, for low-lying zeros, we extend Kadiri's result to all number fields while improving the main constant.
\end{abstract}

\subjclass[2010]{11R42}

\keywords{Dedekind zeta-functions, zero-free regions, explicit formulae}

\maketitle

\section{Introduction}

Suppose throughout that $L$ is a number field with degree $n_L\ge 2$ and discriminant $\Delta_L$; the absolute discriminant of $L$ will be denoted by $d_L = |\Delta_L|$. Let $\zeta_L(s)$ be the Dedekind zeta-function associated to $L$. 

The error term in the Chebotar\"{e}v Density Theorem (CDT), a vast generalisation of the prime number theorem, the prime number theorem for arithmetic progressions, and the prime ideal theorem, is closely connected to a certain sum over the {non-trivial} zeros $\varrho = \beta + i\gamma$, with $0 < \beta < 1$, of $\zeta_L(s)$ via an explicit formula. Recall that a \textit{zero-free region} for $\zeta_L(s)$ is a region in the critical strip $\{s \in\Bbb{C} : 0 < \Re{s} < 1\}$ that contains no non-trivial zeros. In fact, zero-free regions for $\zeta_L(s)$ are central ingredients that enable one to prove effective bounds for the error in the CDT; see \cites{DasThesis, lagarias1977effective} for further details. In fact, any refinements one can make to the zero-free regions for $\zeta_L(s)$ will automatically improve the error in the CDT, among other consequences. For example, \cites{ahn2019explicit, lagarias1979bound, KadiriWong} gave other applications of zero-free regions for $\zeta_L(s)$. In light of these, the purpose of this article is to update the latest explicit zero-free regions for $\zeta_L(s)$. 

There are infinitely many {non-trivial} zeros $\varrho = \beta + i\gamma$ of $\zeta_L(s)$ and the Generalised Riemann Hypothesis (GRH) postulates that every $\varrho$ satisfies $\beta = \frac{1}{2}$. Platt and Trudgian \cite{PlattTrudgianRH} have verified that if $L = \Q$, then the GRH in this setting, namely the {Riemann Hypothesis} (RH), is valid for the region $|\gamma|\leq 3\cdot 10^{12}$. Furthermore, Tollis \cite{Tollis} has verified that the GRH is true for $\zeta_L(\sigma + it)$ in the region $|t|\leq 92$ for number fields $L$ with $n_L=3$ and $d_L\le 239$, and in the region $|t|\leq 40$ for number fields $L$ with $n_L=4$ and $d_L\leq 320$. In general, computations verifying the GRH are limited and only available in some special cases. 

To obtain a deeper understanding about the non-trivial zeros of $\zeta_L(s)$ that fall outside of these regions where the GRH has been verified, we introduce several new zero-free regions for $\zeta_L(s)$ when $L\neq\Q$ (equivalently $n_L\geq 2$). In particular, we establish separate zero-free regions for the non-trivial zeros $\varrho = \beta + i\gamma$ of $\zeta_L(s)$ when $|\gamma| > 1$, $0 < |\gamma | \leq 1$, and $\gamma = 0$.

\begin{remark}
We restrict our attention to number fields $L$ such that $n_L \geq 2$, since $n_L = 1$ implies $L = \Q$, and $\zeta_L(s) = \zeta(s)$ is the Riemann zeta-function, whose zeros have been well-studied. For example, the lowest-lying zero of $\zeta(s)$ is $\frac{1}{2} + 14.13472\ldots i$ and de la Vall{\'e}e Poussin \cite{ValeePoussin} famously proved that $\zeta(\sigma + it)\neq 0$ in the region $t\geq T$ and
\begin{equation}\label{eqn:dlvp}
    \sigma \geq 1-\frac{1}{R\log t},
\end{equation}
where $T$ and $R$ are positive constants. This is commonly referred to as the classical zero-free region for $\zeta(s)$. Over the years, many authors (including Westphal \cite{Westphal}, Ste\v{c}kin \cite{Stechkin1970}, Rosser--Schoenfeld \cite{RosSch}, Kondrat'ev \cite{Kondratev}, Kadiri \cite{kadiri2005region}, and Mossinghoff--Trudgian \cite{MossinghoffTrudgian2015}) have refined these values of $R$. Mossinghoff, Trudgian, and Yang \cite{MossinghoffTrudgianYang} established the latest admissible value for $R$ in \eqref{eqn:dlvp}; these values being $R = 5.558691$ and $T=2$. 
\end{remark}



\subsection*{Zeros with large ordinate}

For $n_L \geq 2$, we can compute absolute constants $(C_1, C_2, C_3, C_4, T)$ such that $\zeta_L(\sigma + it)\neq 0$ in the region
    \begin{equation}\label{eqn:dlvp_nf}
        \sigma \geq 1 - \frac{1}{C_1\log d_L + C_2 n_L\log |t|+ C_3 n_L + C_4} \quad\text{and}\quad
        |t|\geq T .
    \end{equation}
Historically, Lagarias and Odlyzko generalised de la Vall\'{e}e Poussin's approach into the number field setting in \cite{lagarias1977effective}, which requires a non-negative, even, trigonometric polynomial. However, their proof did not use Ste\v{c}kin's key idea from \cite{Stechkin1970} or attempt to find a more favourable trigonometric polynomial. Kadiri generalised Ste\v{c}kin's approach into the number field setting in \cite{kadiri2012explicit} to establish \eqref{eqn:dlvp_nf} with the values $$(C_1, C_2, C_3, C_4, T) = (12.55, 9.69, 3.03, 58.63,1).$$ Further, by choosing a more favourable trigonometric polynomial, inserting new parameters into Kadiri's method, and refining important bounds for certain gamma factors, Lee (the third-named author) proved in \cite{Lee21} that \eqref{eqn:dlvp_nf} is true with the refined constants 
\begin{equation*}
    (C_1,C_2,C_3,C_4,T) = (12.2411, 9.5347, 0.05017, 2.2692,1) .
\end{equation*}
Our first result is presented in Theorem \ref{thm:large_ordinates}, which refines these computations for number fields $L$ with degree $n_L \geq 3$. Most of the refinement in Theorem \ref{thm:large_ordinates} comes from two sources: 
\begin{enumerate}
    \item[(a)] A new bound, Lemma \ref{lem:McCurley_L2_refined}, for  certain gamma factors that refines \cite[Lem.~2]{Mcc}.
    \item[(b)] New choices of non-negative, even, trigonometric polynomials in the method. 
\end{enumerate}
 
\begin{theorem}\label{thm:large_ordinates}
If $L$ is a number field of degree $n_L \geq 3$, then $\zeta_L(\sigma + it)$ is non-zero in the region \eqref{eqn:dlvp_nf} with $$(C_1, C_2, C_3, C_4, T) = (12.21124, 9.54177, -11.59548, 4.57803, 1).$$ 
If $n_L\geq n_0\geq 3$, then further computations (for any $n_0 \leq 21$) are presented in Tables \ref{tab:case1_results_16}, \ref{tab:case1_results_40}, and \ref{tab:case1_results_46}.
\end{theorem}

\begin{remark}
It is not unreasonable for $C_3$ to be negative. To see this, observe that if $n_L = 3$, then $\log{d_L} \geq \log{23}$ and $12.21124 \log{23} - 11.59548\cdot 3 > 3.50183$. Indeed, our refinements are designed to ``waste'' as little as possible, and this negative constant demonstrates that there is very little waste in our method (especially by comparison to previous results). 
\end{remark}

\begin{remark}
If $L$ is a number field of degree $n_L = 2$, then there exists a Dirichlet character $\chi_{L}$ modulo $d_L$ such that $\zeta_L(s)=\zeta(s)L(s,\chi_{L})$, where $L(s,\chi_{L})$ is the Dirichlet $L$-function attached to $\chi_{L}$. By extending the above-mentioned works of Ste\v{c}kin and McCurley, Kadiri \cite{KadiriPhD, KadiriLFunctions} further introduced an extra ``smoothing'' to establish that if $\chi$ is a Dirichlet character modulo $q$, then $L(s,\chi)$ has at most one zero in the region
\begin{equation*}
    \Re{s}\geq 1 - \frac{1}{\mathcal{R}\log(q\max\{1,|\Im(s)|\})} ,
\end{equation*}
with $\mathcal{R}=6.44$ (for all $q\geq 3$) and $\mathcal{R}=5.60$ (for any $3\leq q\leq 400\,000$). Applying this with $q=d_L$, we see that for $n_L=2$, $\zeta_L(\sigma + it)$ has at most one zero in the region
\begin{equation*}
    \sigma \ge 1- \frac{1}{\mathcal{R}\log(d_L\max\{1,|t|\})}.
\end{equation*}
\end{remark}

\subsection*{Zeros with small ordinate} 

If $n_L \geq 2$, then we can compute absolute constants $A>0$ and $A'>0$ such that $\zeta_L(\sigma + it)\neq 0$ in the region
\begin{equation}\label{eqn:dlvp_nf_llz}
    \sigma \geq 1 - \frac{1}{A\log d_L + A' n_L\log (|t|+2)} \quad\text{and}\quad
    |t|\leq 1,
\end{equation}
with the exception of at most one real zero $\beta_1$. This zero, if it exists, is called an \textit{exceptional} zero. 
If $d_L$ is sufficiently large, then Kadiri proved in \cite[Thm.~1.1]{kadiri2012explicit} that $A=12.74$ and $A'=0$ are admissible for the region \eqref{eqn:dlvp_nf_llz}; Lee (the third-named author) refined these computations to $A=12.44$ and $A'=0$ in \cite[Thm.~2]{Lee21}. For all number fields $L\neq \Bbb{Q}$, Ahn and Kwon proved that $\zeta_L(\sigma + it)\neq 0$ in the region \eqref{eqn:dlvp_nf_llz} with $A=A'=29.57$ in \cite[Prop.~6.1]{ahn2019explicit}.

Our next result (Theorem \ref{thm:for|t|<1}) refines the computations in \cite{kadiri2012explicit, Lee21} and removes the ``sufficiently large'' condition that was previously enforced when $A' = 0$. These improvements originated from the above-mentioned sources (a) and (b), as well as refinements to the optimisation process, which is necessary to balance certain parameters that naturally appear in the calculations.  

\begin{theorem}\label{thm:for|t|<1}
If $L$ is a number field of degree $n_L \geq 2$, then $\zeta_L(\sigma + it)$ has at most one zero, namely the exceptional zero (if it exists), in the region \eqref{eqn:dlvp_nf_llz}, with $A' = 0$, and $A$ is given by
\begin{center}
\begin{tabular}{|c|c|c|c|c|c|c|}
\hline
    $n_L$ & $2$ & $3$ & $4$ & $5$ & $6$ & $\geq  7$ \\
    \hline
    $A$ & $16.01983$ & $19.55293$ & $16.72207$ & $13.71235$ & $11.78180$ & $11.51910$ \\
    \hline
\end{tabular} 
\end{center}
\end{theorem}

Recall that the Deuring--Heilbronn phenomenon roughly asserts that if an exceptional zero $\beta_1$ exists, then the zero-free region for $\zeta_L(s)$ can be enlarged. We refer the interested reader to \cite{KadiriWong} for an explicit statement of the phenomenon as well as the references therein. Inspired by this, the following result describes the strongest available zero-free region that would be true, if an exceptional zero $\beta_1$ exists.

\begin{theorem}\label{thm:exp-zero-ZFR}
Let $L$ be a number field with $n_L \geq 2$ and $\mathscr{L} = \log{d_L}$. If a real exceptional zero $\beta_1$ presents in the region $[1 - \frac{\nu}{\mathscr{L}}, 1 )$, then $\zeta_L(\sigma + it)$ is non-zero in the region \eqref{eqn:dlvp_nf_llz}, except for $\sigma + it=\beta_1$, with $A' = 0$, and with admissible values of $A$ for $\nu = 0.05$ given in 
\begin{center}
\begin{tabular}{|c|c|c|c|c|c|c|c|c|c|}
\hline
 $n_L$ & $2$ & $3$ & $4$ & $5$ & $6$ & $7$ & $ 8$ & $ \ge 9$  \\
    \hline
    $A$ & $2.576574$ & $3.316475$ & $2.851720$ & $2.312646$ & $1.960456$  & $1.855168$ & $1.830414$ & $1.806471$ \\
    \hline
\end{tabular} 
\end{center}
and admissible values of $A$ for $\nu = 0.5$ given in
\begin{center}
\begin{tabular}{|c|c|c|c|c|c|c|c|}
\hline
    $n_L$ & $2$ & $3$ & $4$ & $5$ & $6$ & $7$ & $\geq 8$\\
    \hline
    $A$ & $6.036555$ & $8.253321$ & $6.422941$ & $4.668822$ & $3.659933$ & $3.644519$ & $3.576987$ \\
    \hline
\end{tabular} 
\end{center}
\end{theorem}

\begin{remark} (i) In fact, when $\nu =0.05$, we show that the slightly better value $A= 1.806301$ is admissible for $n_L\ge 10$ to establish the first table of Theorem \ref{thm:exp-zero-ZFR}.\\
\noindent (ii) Our argument to establish Theorem \ref{thm:exp-zero-ZFR} in Section \ref{section::exceptional} works for the general situation that an exceptional zero $\beta_1$ presents in the region $[1 - \frac{\nu}{\mathscr{L}},1)$ with $\nu\in(0,\frac{1}{2}]$. The choice $\nu= \frac{1}{20}$ is particularly interesting, because it was assumed in \cite[Sec. 3.4]{KadiriWong} to study the least prime problem in the CDT.
\end{remark}

\subsection*{Zeros on the real line}

For $d_L$ being sufficiently large, Kadiri \cite[p.~146]{kadiri2012explicit} established that $\zeta_L(\sigma)$ admits at most one zero in the region 
\begin{equation}\label{eqn:onrealline}
    \sigma \geq  1 -  \frac{1}{A''\log d_L} ,
\end{equation} 
with $A'' = 1.6110$. The next theorem extends her work to all number fields and removes the ``sufficiently large'' condition.

\begin{theorem}\label{thm:onrealline}
If $L$ is a number field with $n_L \geq 2$, then $\zeta_L(\sigma)$ admits at most one real zero in the region \eqref{eqn:onrealline} with $A''$ given as
\begin{center}
\begin{tabular}{|c|c|c|c|c|c|c|}
\hline
    $n_L$ & $2$ & $3$ & $4$ & $5$ & $6$ & $\geq  7$ \\
    \hline
    $A''$ & $1.61094$ & $1.93173$ & $1.88178$ & $1.69958$ & $1.61857$ & $1.61094$ \\
    \hline
\end{tabular} 
\end{center}
\end{theorem}

Our final result demonstrates the impact of the existence of a real exceptional zero $\beta_1$ on the real zeros:

\begin{theorem}\label{thm:onrealline-exceptional}
Let $L$ be a number field with $n_L \geq 2$ and $\mathscr{L} = \log{d_L}$. If a real exceptional zero $\beta_1$ presents in the region $[1 - \frac{\nu}{\mathscr{L}}, 1 )$,  then $\zeta_L(\sigma)$  is non-vanishing in the region \eqref{eqn:onrealline}, except for $\sigma =\beta_1$, 
with admissible $A''$ for $\nu = 0.05$ given in 
\begin{center}
\begin{tabular}{|c|c|c|c|c|}
\hline
    $n_L$ & $2$ & $3$ & $4$ & $\geq 5$ \\
    \hline
    $A''$ & $0.478632$ & $0.483802$ & $0.480747$ & $0.478632$ \\
    \hline
\end{tabular} 
\end{center}
and with $A''$ for $\nu = 0.5$ given in
\begin{center}
\begin{tabular}{|c|c|c|c|c|c|c|}
\hline
    $n_L$ & $2$ & $3$ & $4$ & $5$ & $6$ & $\geq 7$\\
    \hline
    $A''$ & $1.32086$ & $1.86631$ & $1.77210$ & $1.45550$ & $1.33079$ & $1.32086$\\
    \hline
\end{tabular} 
\end{center}
\end{theorem}

\subsection*{Structure}


In Section \ref{sec:setup}, we introduce several useful notations and preparatory observations which will be required throughout this paper. The proof of Theorem \ref{thm:large_ordinates} will be given in Section \ref{Case-1-proof}. We will prove Theorem \ref{thm:large_ordinates} in Sections \ref{sec:low_lying_complex_zeros} (for complex zeros) and \ref{upper_bound} (for real zeros).

\section{Set-up and preparatory observations}\label{sec:setup}

Our initial set-up is the same as that used by Kadiri and Lee to prove \cite[Thm.~1.1]{kadiri2012explicit} and \cite[Thm.~1]{Lee21}, respectively, following a similar shape to Ste\v{c}kin's argument in \cite{Stechkin1970}. To begin, recall that the Dedekind zeta function $\zeta_L(s)$ associated to the number field $L$ is defined by 
\begin{equation*}
    \zeta_L (s) = \sum_{\mathfrak{a}\neq 0} \N(\mathfrak{a})^{-s} = \prod_{\mathfrak{p}} (1 -\N(\mathfrak{p})^{-s})^{-1},
\end{equation*}
for $\Re{s}>1$. Here, $\N(\mathfrak{a})$ is the norm of $\mathfrak{a}$, the sum is over non-zero integral ideals of $L$, and the product is over prime ideals of the ring of integers $\mathcal{O}_L$ of $L$. It is known that $\zeta_L (s)$ has an analytic continuation to a meromorphic function on $\Bbb{C}$ with only a simple pole at $s = 1$, and its zeros $\rho= \b + i \g$ encode deep arithmetic information of $L$.
The logarithmic derivative of $\zeta_L$ is 
$$
-\frac{\zeta'_L}{\zeta_L}(s)= \sum_{\mathfrak{a}\neq 0} \frac{\Lambda(\mathfrak{a})}{\N(\mathfrak{a})^s},
$$
where $\Lambda(\mathfrak{a})$ is the number field analogue of the von Mangoldt function:
\begin{equation*}
    \Lambda(\mathfrak{a}) = 
    \begin{cases}
        \log \N(\mathfrak{p}) &\text{if $\mathfrak{a} = \mathfrak{p}^k$ for some prime ideal $\mathfrak{p}$ of $\mathcal{O}_L$ and $k\ge 1$;} \\
        0 &\text{otherwise.}
    \end{cases}
\end{equation*}

Let $r_1$ and $r_2$ be the number of real and complex places (respectively) of $L$, and note that $n_L = r_1 +2r_2$. The completed zeta function $\xi_L(s) $ is
\begin{equation*}\label{def-xi}
\xi_L(s)= s(s-1) d_L^{s/2} \g_L (s) \zeta_L (s),
\end{equation*}
where 
\begin{equation}\label{eqn:gammaL}
    \gamma_L (s) = \Big(\pi^{-\frac{s+1}{2}} \Gamma \Big( \frac{s+1}{2}\Big) \Big)^{r_2} \Big( \pi^{-\frac{s}{2}} \Gamma \Big( \frac{s}{2}\Big)\Big)^{r_1 + r_2}.
\end{equation}
Next, we let $t\in\R$. The following definitions hold for the rest of this paper:
\begin{itemize}\label{2.2}
    \item $\kappa = \frac{1}{\sqrt{5}}$,
    \item $s_k = \sigma + ikt$ such that $k\in \mathbb{N}$, $1<\sigma < 1 + \varepsilon$ for some $0 < \varepsilon \leq 0.15$, and
    \item $s'_k = \sigma_1 + ikt$ such that $\sigma_1 = \frac{1 + \sqrt{1 + 4\sigma^2}}{2}$.
\end{itemize}
For convenience, we will write $\sigma_1(a)$ to denote the value of $\sigma_1$ at $\sigma = a$. To prove our results, we isolate a non-trivial zero $\varrho_0 = \beta_0 + i\gamma_0$ of $\zeta_L(s)$ such that $\beta_0 > 1 - \varepsilon \geq 0.85$, and choose a polynomial $p_n(\varphi)$ from the class $P_n$ of non-negative, even, trigonometric polynomials of degree $n \geq 2$, which is defined by
\begin{equation*}\label{poly}
    P_n:=\Big\{p_n(\varphi)=\sum_{k=0}^n a_k\cos (k\varphi):p_n(\varphi)\geq 0\text{ for all }\varphi\text{, }a_k\geq 0\text{ and }a_0<a_1\Big\}.
\end{equation*}
Now, consider the function
\begin{equation*}
    S(\sigma,t) = \sum_{k=0}^{n} a_kf_L(\sigma, kt),
\end{equation*}
in which
\begin{align*}
    f_L(\sigma,kt)
    &= -\Re\Big(\frac{\zeta'_L}{\zeta_L}(s_k) - \kappa \frac{\zeta'_L}{\zeta_L}({s'_k})\Big)= \sum_{0\neq\mathfrak{p}\subset\mathcal{O}_L} \Lambda(\mathfrak{p})(N(\mathfrak{p})^{-\sigma}-\kappa N(\mathfrak{p})^{-\sigma_1}) \cos(k t\log(N(\mathfrak{p})),\label{eqn:useful1}
\end{align*}
because $\sigma > 1$. With the choices for $\sigma_1$ and $\kappa$ that we made, the non-negativity of $p_n(\varphi)$ implies that $S(\sigma, t) \geq 0$. On the other hand, the explicit formula in \cite[Eqn. (8.3)]{lagarias1977effective} tells us that
\begin{equation}\label{eqn:explicit_formula}
    -\Re\frac{\zeta^{\prime}_{L}}{\zeta_{L}}(s)=-\sum_{\rho\in Z_{L}}\Re\frac{1}{s-\rho}+\frac{\log{d_L}}{2}+\Re\frac{1}{s}+\Re\frac{1}{s-1}+\Re\frac{\gamma^{\prime}_{L}}{\gamma_{L}}(s) ,
\end{equation}
where $\gamma_L$ was defined in \eqref{eqn:gammaL} and $Z_L$ denotes the set of non-trivial zeros of $\zeta_L$. It follows from \eqref{eqn:explicit_formula} that
\begin{equation}\label{eqn:imp}
    0\leq S(\sigma,t)\leq S_1+S_2+S_3+S_4,
\end{equation}
where $F(s,z) = \Re(\frac{1}{s-z}+\frac{1}{s-1+\bar{z}})$, 
\begin{align*}
    S_1 &= -\sum_{k=0}^{n} a_k\sum_{\varrho\in Z_L}\Re\Big(\frac{1}{s_k-\varrho} - \frac{\kappa}{{s'_k}-\varrho}\Big),\\
    S_2 &= \frac{1-\kappa}{2} \Big(\sum_{k=0}^{n} a_k\Big) \log d_L,\\
    S_3 &= \sum_{k=0}^{n} a_k \left(F(s_k,1) -\kappa F({s'_k},1)\right),\text{ and}\\
    S_4 &= \sum_{k=0}^{n} a_k \Re\Big(\frac{\gamma_L'(s_k)}{\gamma_L(s_k)} - \kappa \frac{\gamma_L'({s'_k})}{\gamma_L({s'_k})}\Big).
\end{align*}
We shall denote $\mathscr{L} = \log{d_L}$ throughout, and we will consider the following five cases separately:
\begin{description}
    \item[Case 1] $\gamma_0 > 1$
    \item[Case 2] $\frac{d_2}{\mathscr{L}} < \gamma_0 \leq 1$
    \item[Case 3] $\frac{d_1}{\mathscr{L}} < \gamma_0 \leq \frac{d_2}{\mathscr{L}}$
    \item[Case 4] $0 < \gamma_0 \leq \frac{d_1}{\mathscr{L}}$
    \item[Case 5] $\gamma_0 = 0$
\end{description}
Henceforth, we shall further assume $t = \gamma_0$ (we can use these interchangeably) and in the fifth case, we consider two real zeros $\rho_1 = \beta_1 + i \gamma_1$ and $\rho_2 = \beta_2 + i \gamma_2$ such that $\beta_1 \leq \beta_2$. We consider each case in the upcoming sections. For convenience, we also introduce some intermediary results in Sections \ref{ssec:mindiscs}-\ref{ssec:gammafactors}, which we will require later.

\begin{remark} (i)
The rationale for the split points we have chosen in Cases 1-5 is the following. First, we will use the same method to tackle Cases 1-2, although the ranges affect the appearance of the final outcome. As a result, we obtain $d_2$, which gives the best values in Case 2. Second, we handle Cases 3-4 together while using the value of $d_2$ obtained from Case 2 to find the value of $d_1$ that gives the best maximum values from Cases 3-4. Case 3 employs a distinct method from Cases 1-2, and Case 4 is of the Stark style.\\
\noindent (ii)
Our eventual bounds for Cases 1-3 will depend on the choice of the polynomial, whereas our bounds for Cases 4-5 will not depend on this choice.
\end{remark}

\subsection{Minimum discriminants}\label{ssec:mindiscs}

Suppose that $L$ is a number field with degree $n_L$ and absolute discriminant $d_L$. Table \ref{Table-imp} presents lower bounds for the size of $d_L$ at given choices of $n_L$. These computations are imported from the appendix in \cite{KadiriWong}.

\begin{table}[ht]
\begin{tabular}{|cc|cc|cc|cc|}
    \hline
    $n_L$ & $d_\text{min}(n_L)$ & $n_L$ & $d_\text{min}(n_L)$ & $n_L$ & $d_\text{min}(n_L)$ & $n_L$ & $d_\text{min}(n_L)$ \\
    \hline
    2 & $3$ & 7 &$ 184\,607$ & 12 & $2.74\cdot 10^{10}$  & 17 & $3.70\cdot 10^{16}$    \\
    3 & $23$ & 8 & $1\,257\,728$ &   13 & $7.56\cdot 10^{11}$   & 18 & $2.73\cdot 10^{17}$   \\
    4 & $117$ & 9 & $2.29\cdot 10^7$   & 14 & $5.43\cdot 10^{12}$  & 19 & $9.03\cdot 10^{18}$     \\
    5 & $1\,609$  &  10 & $1.56\cdot 10^8$  & 15 & $1.61\cdot 10^{14}$   & 20 & $6.74\cdot 10^{19}$     \\
    6 & $9\,747$  &11 &$3.91\cdot 10^9$   & 16 &$1.17 \cdot 10^{15}$  & $\geq 21$ & $10^{n_L}$ \\  
    \hline
\end{tabular}
\caption{Computations for $d_{\text{min}}(n_L)$ such that $d_L \geq d_{\text{min}}(n_L)$ for every number field of a given degree $n_L \geq 2$.}
\label{Table-imp}
\end{table}

\subsection{Bounds for gamma factors}\label{ssec:gammafactors}

The following technical lemma refines \cite[Lem.~2]{Mcc}. We will use this to refine the bounds on certain gamma factors that are important.

\begin{lemma}\label{lem:McCurley_L2_refined}
Let $1 <\sigma < 1+\varepsilon$ for some $0<\varepsilon\leq 0.15$, $ k\geq 1$, $\delta\in\{0,1\}$, $s_k=\sigma+ikt$,
and $s_k' = \sigma_1+ ikt$. We have
\begin{align*}
    &\frac{1}{2}\Re\Big(
    \frac{\Gamma'}{\Gamma}\Big(\frac{s_k+\delta}{2}\Big)-\kappa\frac{\Gamma'}{\Gamma}\Big(\frac{s_k'+\delta}{2}\Big)
    \Big)\\
    &\qquad=\frac{1-\kappa}{2}\log \frac{kt}{2}+\Xi(\sigma, k, t, \delta)
    +\frac{\theta_1}{6}\Big(\frac{1}{(\sigma+\delta)^2+(kt)^2}\Big)
    +\frac{\theta_2\kappa}{6}\Big(\frac{1}{(\sigma_1+\delta)^2+(kt)^2}\Big),
\end{align*}
for some $|\theta_i|\leq 1$, where
\begin{align*}
    \Xi(\sigma,k,t,\delta)&:=\frac{1}{4}\log \Big(1+\Big(\frac{\sigma+\delta}{kt}\Big)^2\Big)-\frac{\kappa}{4} \log \Big(1+\Big(\frac{\sigma_1+\delta}{kt}\Big)^2\Big)\\
    &\qquad\qquad - \frac{\sigma +\delta}{2((\sigma+\delta)^2+(kt)^2)}
    + \kappa\frac{\sigma_1 +\delta}{2((\sigma_1+\delta)^2+(kt)^2)}.
\end{align*}
\end{lemma}

\begin{proof}
We know that
\begin{align}\label{eq: estimate_gamma}
    \frac{\Gamma'}{\Gamma}(s)
    = \log{s} - \frac{1}{2s} - 2 \int_0^{\infty} \frac{v}{(s^2 + v^2)(e^{2\pi v}-1)}\,dv ;
\end{align}
see \cite[p.~251]{WhittakerWatson}. It follows from \eqref{eq: estimate_gamma} that 
\begin{equation}\label{eq: Re(digamma)}
\begin{split}
    \Re \frac{\Gamma'}{\Gamma}\Big(\frac{s_k+\delta}{2}\Big)     &= \log\Big|\frac{s_k+\delta}{2}\Big| +
    \Re\Big(\frac{1}{s_k+\delta}\Big)  - 2\Re
    \int_0^{\infty} \frac{v}{((\frac{s_k+\delta}{2})^2 + v^2)(e^{2\pi v}-1)}dv.
\end{split}
\end{equation}
Moreover, we have
\begin{align*}
    \Re\int_0^{\infty} \frac{v}{(z^2 + v^2)(e^{2\pi v}-1)}dv
    \leq \Big| \int_0^{\infty} \frac{v}{(z^2 + v^2)(e^{2\pi v}-1)}\,dv \Big|  ,
\end{align*}
which is
$$
\le \int_0^{\infty} \Big|\frac{v}{(z^2 + v^2)(e^{2\pi v}-1)}\Big|\,dv
\leq \frac{1}{|z^2|}\int_0^{\infty} \frac{v}{e^{2\pi v}-1}\,dv 
= \frac{1}{24|z|^2} .
$$
Thence, we can write
\begin{equation*}
    \Re\int_0^{\infty} \frac{v}{(z^2 + v^2)(e^{2\pi v}-1)}dv = \frac{\theta}{24|z|^2} ,
\end{equation*}
wherein $|\theta|\leq 1$, and so we can insert $z = \frac{s_k + \delta}{2}$ into this observation to obtain
\begin{equation*}
    - 2 \Re \int_0^{\infty} \frac{v}{((\frac{s_k + \delta}{2})^2 + v^2)(e^{2\pi v}-1)}dv
    = \frac{\theta}{3 |s_k + \delta|^2}
    = \frac{\theta}{3 ((\sigma + \delta)^2 + k^2 t^2)} .
\end{equation*}
In addition, we know
\begin{equation*}
    \log\Big|\frac{s_k+\delta}{2}\Big| 
    = \log{\frac{kt}{2}} + \frac{1}{2} \log\Big(1 + \Big(\frac{\sigma + \delta}{kt}\Big)^2\Big)
    \text{ and }
    \Re\Big(\frac{1}{s_k+\delta}\Big)
    = \frac{\sigma + \delta}{(\sigma + \delta)^2 + k^2 t^2} .
\end{equation*}
Therefore, we insert these observations into \eqref{eq: Re(digamma)} to yield the expected bounds.
\end{proof}

Using Lemma \ref{lem:McCurley_L2_refined}, we can also prove the following result.

\begin{lemma}
\label{lem:McCurley_L1_refined}
Let $1<\sigma <1.15$, $|t| \leq 1$, $1\leq k \leq 46$, $\delta\in \{0,1\}$, $s_k=\sigma+ikt$, and $s_k'=\sigma_1+ikt$. Then 
\begin{align}\label{lem2.2-eq}
    \frac{1}{2}\Re\Big(
    \frac{\Gamma'}{\Gamma}\Big(\frac{s_k+\delta}{2}\Big)-\kappa\frac{\Gamma'}{\Gamma}\Big(\frac{s_k'+\delta}{2}\Big)
    \Big) 
    \le\mathfrak{d}(\delta,k),
\end{align}
where values for $\mathfrak{d}(\delta,k)$ with are presented in Table \ref{tab:d(k)}.
\end{lemma}

\begin{proof}
Using Lemma \ref{lem:McCurley_L2_refined}, we can bound the left of \eqref{lem2.2-eq} above by
\begin{align*}
 &\frac{1}{4}\log \Big(\frac{(\sigma+\delta)^2+(kt)^2}{4}\Big)-\frac{\kappa}{4}\log \Big(\frac{(\sigma_1+\delta)^2+(kt)^2}{4}\Big) -\frac{1}{2}\Big(\frac{\sigma+\delta}{(\sigma+\delta)^2 
 +(kt)^2}\Big) \\
    &+\frac{\kappa}{2}\Big(\frac{\sigma_1+\delta}{(\sigma_1+\delta)^2+(kt)^2}\Big) +\frac{1}{6}\Big(\frac{1}{(\sigma+\delta)^2+(kt)^2}\Big)
    +\frac{\kappa}{6}\Big(\frac{1}{(\sigma_1+\delta)^2+(kt)^2}\Big) .
\end{align*}
Optimising these values, we get the explicit values presented in Table \ref{tab:d(k)}.
\end{proof}

\section{Case 1: large ordinates}\label{Case-1-proof}

In this section, bring forward all of the notations and set-up from Section \ref{sec:setup}. Our first result, namely Theorem \ref{thm:large_ordinates} will follow by considering Case 1, that is $\gamma_0 > 1$. In particular, we will bound $S_1$, $S_3$, and $S_4$ in Section \ref{ssec:S3S4}, make choices for our polynomial in Section \ref{ssec:polynomial_choice}, and combine observations to prove Theorem \ref{thm:large_ordinates} in Section \ref{ssec:comps}. Note that $S_2$ will be computed directly, so there is no need to bound it.

\subsection{Bounds for $S_i$}\label{ssec:S3S4}

Kadiri has shown that 
\begin{equation}\label{eqn:S1}
    S_1 \leq - \frac{a_1}{\sigma - \beta_0} ,
\end{equation}
where $\beta_0$ is the real part of the non-trivial zero $\varrho_0$ that we isolated earlier; see \cite[Sec.~2.2]{kadiri2012explicit}. 
Next, we bound $S_3$ and $S_4$ in the following lemmas.

\begin{lemma}\label{lem:S3}
Suppose that $h(\sigma) = \frac{1}{\sigma} - \frac{\kappa}{\sigma_1} - \frac{\kappa}{\sigma_1 - 1}$ and
\begin{align*}
    \Sigma_k(\sigma,t)
    &:= F(\sigma + ikt, 1) - \kappa F(\sigma_1 + ikt, 1)\\
    &= \frac{\sigma}{\sigma^2 + k^2 t^2} + \frac{\sigma - 1}{(\sigma - 1)^2 + k^2 t^2} - \kappa\frac{\sigma_1}{{\sigma_1}^2 + k^2 t^2} - \kappa \frac{\sigma_1 - 1}{(\sigma_1 - 1)^2 + k^2 t^2} .
\end{align*}
If $\alpha_{\varepsilon} = h(1 + \varepsilon) < 0.021467$, then 
\begin{equation*}
    S_3 \leq a_0\Big(\frac{1}{\sigma - 1} + \alpha_{\varepsilon}\Big) + \sum_{k = 1}^{n} a_k \Sigma_k(1+\varepsilon,1) .
\end{equation*}
\end{lemma}

\begin{proof}
Consider the cases $k=0$ and $k>0$ separately, and follow the arguments in \cite[Sec.~2.2]{Lee21}.
\end{proof}

\begin{lemma}\label{lem:S4}
Suppose that
\begin{equation*}
    \frac{1}{2}\max_{\delta\in\{0,1\}}\Big\{\Re\Big(\frac{\Gamma'}{\Gamma}\Big(\frac{\sigma + \delta}{2}\Big) - \kappa \frac{\Gamma'}{\Gamma}\Big(\frac{\sigma_1 + \delta}{2}\Big)\Big)\Big\} \leq d_{\varepsilon}(0),
\end{equation*}
wherein $d_{\varepsilon}(0)$ is the maximum of the functions such that $\sigma = 1 + \varepsilon$. Moreover, we write $t > T_0 \geq 1$,
\begin{equation*}
    \mathcal{S}_1(k,\varepsilon) = \max_{\delta\in\{0,1\}}\big\{\mathcal{C}_1(k,\delta,\varepsilon)\big\},
    \quad\text{and}\quad
    \mathcal{S}_2(k,\varepsilon) = \max_{\delta\in\{0,1\}}\big\{\mathcal{C}_2(k,\delta,\varepsilon)\big\}
\end{equation*}
such that
\begin{align*}
    \mathcal{C}_1(k,\delta,\varepsilon)
    &= \frac{1 - \kappa}{2}\log \frac{k}{2} + \Xi_2(1+\varepsilon,k,T_0,\delta)\\
    & + \frac{1}{6}\Big(\frac{1}{(1-\delta)^2 + k^2 T_0^2}\Big) +\frac{\kappa}{6} \Big(\frac{1}{(\sigma_1(1) -\delta)^2 + k^2 T_0^2}\Big), \\
    \mathcal{C}_2(k,\delta,\varepsilon)
    &= \frac{1 - \kappa}{2}\log \frac{k}{2} + \mathcal{A}(k,T_0,\delta,\varepsilon)\\
    & + \frac{1}{6}\Big(\frac{1}{(1-\delta)^2 + k^2 T_0^2}\Big) +\frac{\kappa}{6} \Big(\frac{1}{(\sigma_1(1) -\delta)^2 + k^2 T_0^2}\Big) ,
         \end{align*}
    where     
    \begin{align*}
    \Xi_2(\sigma,k,t,\delta) &= \frac{1}{4}\log\!\Big(1 + \Big(\frac{\sigma + \delta}{kt}\Big)^2\Big) - \frac{\kappa}{4}\log\!\Big(1 + \Big(\frac{\sigma_1 + \delta}{kt}\Big)^2\Big),\\
           \Xi(\sigma,k,t,\delta)
    &= \Xi_2(\sigma,k,t,\delta) - \frac{\sigma + \delta}{2((\sigma + \delta)^2 + k^2t^2)}
    + \frac{\kappa}{2} \frac{\sigma_1 + \delta}{(\sigma_1 + \delta)^2 + k^2t^2} , \text{ and}\\
    \mathcal{A}(k,t,\delta,\varepsilon) &=
    \begin{cases}
        0&\text{if }\delta = 0\text{ or }\delta = 1\text{ and }k\not\in\{1,2\};\\
        \Xi(1+\varepsilon,k,t,1) &\text{if }\delta = 1\text{ and }k = 1;\\
        \Xi(1.15,k,t,1)&\text{if }\delta = 1\text{ and }k=2.
\end{cases}
\end{align*}
If $\mathcal{S}(k,\varepsilon)=\min\{\mathcal{S}_1(k,\varepsilon), \mathcal{S}_2(k,\varepsilon)\}$, then
\begin{equation*}
    \Re\Big(\frac{\gamma_L'(s_k)}{\gamma_L(s_k)} - \kappa \frac{\gamma_L'({s'_k})}{\gamma_L({s'_k})}\Big) \leq 
\begin{cases}
    n_L\left(d_{\varepsilon}(0) - \frac{1-\kappa}{2} \log\pi \right)&\text{if }k=0;\\
    n_L\left(\frac{1 - \kappa}{2} (\log{t} + \log(\frac{k}{\pi}) ) + \mathcal{S}(k,\varepsilon) \right)&\text{if }k\neq 0.
\end{cases}
\end{equation*}
Consequently, we have
\begin{align*}
    S_4 
    &\leq a_0 n_L\Big(d_{\varepsilon}(0) -\frac{1-\kappa}{2} \log\pi \Big) + \sum_{k = 1}^{n} a_k n_L\Big(\frac{1 - \kappa}{2}\Big(\log t + \log\Big(\frac{k}{\pi}\Big)\Big) + \mathcal{S}(k,\varepsilon) \Big).
\end{align*}
\end{lemma}

\begin{proof}
Consider the cases $k=0$ and $k>0$ separately and follow the arguments laid out in \cite[Sec.~2.3]{Lee21}, using Lemma \ref{lem:McCurley_L2_refined} instead of \cite[Lem.~2]{Mcc}.
\end{proof}

\subsection{Completing the argument}\label{ssec:completion}

Suppose that $r>0$, and $\sigma$ is chosen such that $$\sigma - 1 = r(1-\beta_0).$$ Insert the upper bounds for $S_i$ from \eqref{eqn:S1}, Lemma \ref{lem:S3}, and Lemma \ref{lem:S4} into \eqref{eqn:imp} and rearrange the resulting inequality to see 
\begin{equation}\label{eqn:useful2}
    \beta_0 \leq 1 - \frac{\frac{a_1}{1+r} - \frac{a_0}{r}}{c_1\log d_L + c_2 n_L \log t + c_3 n_L + c_4},
\end{equation}
where
\begin{align*}
    c_1 &= \frac{1-\kappa}{2}\sum_{k=0}^{n} a_k,\\
    c_2 &= \frac{1-\kappa}{2}\sum_{k=1}^{n} a_k,\\
    c_3 &= a_0 \Big(d_{\varepsilon}(0) - \frac{1-\kappa}{2}\log \pi\Big)+ \sum_{k=1}^{n} a_k\Big(\frac{1 - \kappa}{2}\log\Big(\frac{k}{\pi}\Big) + \mathcal{S}(k,\varepsilon) \Big),\text{ and}\\
    c_4 &= \alpha_{\varepsilon} a_0 + \sum_{k=1}^{n} a_k \Sigma_k(1+\varepsilon,1) .
\end{align*}
The maximum value of $\frac{a_1}{1+r} - \frac{a_0}{r}$ occurs at $r = \frac{\sqrt{a_0}}{\sqrt{a_1} - \sqrt{a_0}}$. Therefore, dividing the numerator and denominator of \eqref{eqn:useful2} by
$$M=\frac{a_1}{1+\frac{\sqrt{a_0}}{\sqrt{a_1} - \sqrt{a_0}}} - \frac{a_0}{\frac{\sqrt{a_0}}{\sqrt{a_1} - \sqrt{a_0}}},$$
we see that
\begin{equation*}
    \beta_0\leq 1 - \frac{1}{\frac{c_1}{M}\log{d_L} + \frac{c_2}{M} n_L \log{t} + \frac{c_3}{M} n_L + \frac{c_4}{M}}
    \quad\text{for all}\quad
    |\gamma_0| > 1 ;
\end{equation*}
this is \eqref{eqn:dlvp_nf} with
\begin{equation}\label{eqn:general}
    (C_1,C_2,C_3,C_4,T) = \left(\frac{c_1}{M}, \frac{c_2}{M}, \frac{c_3}{M}, \frac{c_4}{M},1\right) .
\end{equation}
Clearly, the choice of $\varepsilon$ influences the definitions of $c_3$ and $c_4$, in that these values are minimised when $\varepsilon$ is minimised. Therefore, we should choose $\varepsilon$ to be as small as possible, but also enforce the condition
\begin{equation}\label{eqn:esp_choice}
    \varepsilon > \frac{1}{\frac{c_1}{M}\log{d_L} + \frac{c_3}{M} n_L + \frac{c_4}{M}},
\end{equation}
which will prevent the condition $\beta_0 > 1 - \varepsilon$ from being violated. So, to find the latest admissible constants in a zero-free region of the form \eqref{eqn:dlvp_nf} for number fields $L$ with
$n_L \geq n_0 \geq 3$, all that remains is to choose $\varepsilon$ and the polynomial appropriately. To this end, we note that
\begin{equation*}
    \frac{1}{\frac{c_1}{M}\log{d_L} + \frac{c_3}{M} n_L + \frac{c_4}{M}} \leq \frac{1}{\frac{c_1}{M}\log{d_\text{min}(n_0)} + \frac{c_3}{M} n_0 + \frac{c_4}{M}},
\end{equation*}
where $d_\text{min}(n_0)$ is the smallest permissible discriminant for a number field with $n_L \geq n_0$. Recall that admissible values for $d_\text{min}(n_0)$ are presented in Table \ref{Table-imp}.

\begin{remark}
Alternatively, one could choose $\varepsilon$ arbitrarily, but then the result would only be valid when $t$ satisfies $|t|\geq T$ such that 
\begin{equation*}
    \varepsilon > \frac{1}{\frac{c_1}{M}\log{d_\text{min}(n_0)} + \frac{c_2}{M} n_0 \log{T} + \frac{c_3}{M} n_0 + \frac{c_4}{M}} .
\end{equation*}
For a fixed choice of $\varepsilon$, it is an easy problem to find an explicit value for $T$ such that this would hold for all $|t| \geq T$, and hence we opt to leave this as an exercise for the reader.
\end{remark}

\subsection{Candidates for the choice of polynomial}\label{ssec:polynomial_choice}

Kadiri chose a degree $n=4$ polynomial to prove \eqref{eqn:dlvp_nf} with $(C_1, C_2, C_3, C_4, T) = (12.55, 9.69, 3.03, 58.63,1)$ in \cite{kadiri2012explicit}. Lee used the degree $n=16$ polynomial from \cite{MossinghoffTrudgian2015} in \cite{Lee21} to establish \eqref{eqn:dlvp_nf} with
\begin{equation*}
    (C_1,C_2,C_3,C_4,T) = (12.2411, 9.5347, 0.05017, 2.2692,1).
\end{equation*}
This refinement to Kadiri's method was one of the biggest sources of refinement in \cite{Lee21}, in part because the polynomial was chosen using simulated annealing to improve the latest zero-free region for the Riemann zeta-function at the time. Prior to this, Kondrat\cprime ev had used a degree $n=8$ polynomial in \cite{Kondratev} to beat Ste\v{c}kin's result in \cite{Stechkin1970}. Recently, Mossinghoff, Trudgian, and Yang have unveiled higher degree polynomials (of degree $n\in\{40, 46\}$) in \cite{MossinghoffTrudgianYang}, which were also chosen using simulated annealing to refine the latest zero-free regions for the Riemann zeta-function. The coefficients of these polynomials are presented in Tables \ref{table:coefficients2}-\ref{table:coefficients3}. It seems natural, therefore, that these polynomials can also be used to refine the result in \cite{Lee21}.

\subsection{Computations}\label{ssec:comps}

Now, having a selection of candidate polynomials in hand, we can describe our algorithm to compute new admissible values for $C_i$ in \eqref{eqn:dlvp_nf} with $T=1$ and $n_L\geq n_0\geq 3$. That is, fixing a choice of polynomial and $n_0$, we then follow this straightforward process:

\begin{algorithmic}
\State $\varepsilon \gets 0.0001$
\While{$\left(\frac{c_1}{M}\log{d_\text{min}(n_0)} + \frac{c_3(\varepsilon)}{M} n_0 + \frac{c_4(\varepsilon)}{M}\right)^{-1} \geq \varepsilon$}
    \State $\varepsilon \gets \varepsilon + 0.0001$
\EndWhile
\end{algorithmic}

\noindent The outcome of this process will be the least $\varepsilon$ (up to four decimal places) such that \eqref{eqn:esp_choice} is satisfied. Once this choice of $\varepsilon$ has been determined, insert it into \eqref{eqn:general} to yield admissible computations for $C_i$ such that \eqref{eqn:dlvp_nf} is true with $T=1$.

We computed values following the preceding logic and using the polynomials $p_n$ of degree $n\in\{8,16,40,46\}$ from Kondrat\cprime ev \cite{Kondratev}, Mossinghoff--Trudgian \cite{MossinghoffTrudgian2015}, and Mossingoff--Trudgian--Yang \cite{MossinghoffTrudgianYang} (whose coefficients are presented in Tables \ref{table:coefficients2}-\ref{table:coefficients3}). In the end, there were two outcomes that might be considered the ``best'', depending on whether the reader places more importance on the constant $C_1$ or $C_2$ being minimised. If it is more important that $C_1$ is minimal, then we determined that $p_{46}$ is the best polynomial to choose. On the other hand, if it is more important that $C_2$ is minimal, then we determined that $p_{16}$ is the best polynomial to choose. Our results under these two choices (and $p_{40}$ for comparison) are presented in Tables \ref{tab:case1_results_16}, \ref{tab:case1_results_40}, and \ref{tab:case1_results_46}.

\section{Cases 2-4: complex zeros with small ordinates}\label{sec:low_lying_complex_zeros}

Bring forward all of the notations and set-up from Section \ref{sec:setup}. In this section, we introduce and deal with Cases 2-4, those which encapsulate $0 < |\gamma_0| \leq 1$. For convenience, suppose that
\begin{equation}\label{def-r&c}
    \sigma - 1 = \frac{r}{\mathscr{L}}
    \quad\text{and}\quad
    1 - \beta_0 = \frac{c}{\mathscr{L}}.
\end{equation}
Recall that we are working over $1<\sigma \leq 1 + \varepsilon$ for some $0 < \varepsilon < 0.15$, so $\mathscr{L}\ge\mathscr{L}_0$ and we naturally assert $$0< r < 0.15\mathscr{L}_0.$$

\subsection{Preliminary observations}

We require the following results from \cite[Eqns.~(2.21), (2.23)-(2.27), (2.32), and (2.33)]{kadiri2012explicit}:

\begin{lemma}\label{Lemma 2.3}
Let $\varrho_0 = \beta_0 + i \gamma_0$ be a non-trivial zero of $\zeta_L(s)$ such that $\beta_0 \geq 0.85$ and $\gamma_0 \geq 0$. For $k=1$ in Cases $2$ and $3$, one has
    \begin{equation}{\label{2.21}}
        \sideset{-}{^{\prime}}\sum_{\substack{\varrho\in Z_{L} \\ \beta \geq 1/2}} \left(F(\sigma + i\gamma_0,\varrho) - \kappa F(\sigma_1 + i\gamma_0,\varrho)\right) \leq  -\frac{1}{\sigma - \beta_0} .
    \end{equation}
    Let $\alpha_1 = 1.3951$. For $k \ne 1$ and $\gamma_0 \leq 1$ in Case $2$, one has
    \begin{equation}\label{2.23}
         \sideset{-}{^{\prime}}\sum_{\substack{\varrho\in Z_{L} \\ \beta \geq 1/2}} \left(F(\sigma+ ik\gamma_0,\varrho) - \kappa F(\sigma_1 + ik\gamma_0,\varrho)\right) \leq - \frac{\sigma- \beta_0}{(\sigma - \beta_0)^2 + (k-1)^2\gamma_0^2} + \alpha_1.
    \end{equation}
    For $k\ne 1$ and $\gamma_0 \leq 1$ in Case $3$, one has 
    \begin{align}\label{2.24}
     \begin{split}
        &\sideset{-}{^{\prime}}\sum_{\substack{\varrho\in Z_{L} \\ \beta \geq 1/2}} \left(F(\sigma + ik\gamma_0,\varrho) - \kappa F(\sigma_1+ ik\gamma_0 ,\varrho)\right) \\
        &\qquad\qquad\leq -\frac{\sigma - \beta_0}{(\sigma - \beta_0)^2 + (k-1)^2\gamma_0^2} -\frac{\sigma - \beta_0}{(\sigma - \beta_0)^2 + (k+1)^2\gamma_0^2} + 2\alpha_1 .        
       \end{split} 
   \end{align}
    Finally, for $\gamma_0 \leq 1$ in Cases $4$ and $5$, one has
    \begin{equation}\label{2.25}
        \sideset{-}{^{\prime}}\sum_{\substack{\varrho\in Z_{L} \\ \beta \geq 1/2}} \left(F(\sigma,\varrho) - \kappa F(\sigma_1,\varrho)\right) \leq -2 \frac{\sigma- \beta_0}{(\sigma - \beta_0)^2 + \gamma_0^2} + 2 \alpha_1.
    \end{equation}
\end{lemma}

\begin{lemma}\label{Lemma2.4}
Set $\alpha_2 = 0.0215$ and $\alpha_3 = 1.5166.$   
Then one has
    \begin{equation}\label{2.26}
    F(\sigma,1) - \kappa F(\sigma_1,1) \leq \frac{1}{\sigma-1} + \alpha_2   .    
    \end{equation}
    If $k \geq 1$ and $0 < \gamma_0 < 1$, then
    \begin{equation}\label{2.27}
        F(\sigma + ik\gamma_0,1) - \kappa F(\sigma_1 + ik\gamma_0,1) \leq \frac{\sigma-1}{(\sigma-1)^2 + (k\gamma_0)^2} + \alpha_3.
    \end{equation}
\end{lemma}

\begin{lemma}\label{Lemma 2.6}
For $k = 0$ or $\gamma_0 = 0$, one has
    \begin{equation}\label{2.32}
      \Re \Big(\frac{\gamma^{\prime}_{L}}{\gamma_{L}}(\sigma)-\kappa\frac{\gamma^{\prime}_{L}}{\gamma_{L}}(\sigma_{1}) \Big)
      \leq \Big( d(0) - \frac{(1-\kappa) \log \pi}{2} \Big) n_L,
    \end{equation}
where $d(0) = -0.0512$. 
If $0 < \gamma_0 \leq 1$ and $k \geq 1$, then
    \begin{equation}\label{2.33}
        \Re \Big(\frac{\gamma^{\prime}_{L}}{\gamma_{L}}(\sigma + ik\gamma_0)-\kappa\frac{\gamma^{\prime}_{L}}{\gamma_{L}}(\sigma_{1} + ik\gamma_0) \Big)
        \leq \Big(d(k) - \frac{(1-\kappa) \log{\pi}}{2} \Big) n_L,    
    \end{equation}
where $d(k) = \max \{ \mathfrak{d}(0,k), \mathfrak{d}(1,k) \}$ with $\mathfrak{d}(\cdot,k)$ the same as in Lemma \ref{lem:McCurley_L1_refined}. (Explicit values of $d(k)$, for $k\le 46$, can be calculated using Table \ref{tab:d(k)}.)
\end{lemma}

\subsection{Case 2}\label{case2}

To begin, let $r$ and $c$ be taken as in \eqref{def-r&c}. Suppose further that
\begin{equation}\label{conditioncase2}
    0 < \frac{a_0}{a_1 - a_0}c < r < 1
    \quad\text{and}\quad 
    d_2 > \frac{\sqrt{r(r+c)}}{2} .
\end{equation}
The symmetry of the zeros with respect to the critical line enables one to write
\begin{align*}
    -\sum_{\rho = \beta+i\gamma \in Z_{L}}\Re\Big(\frac{1}{s_k-\rho}-\frac{\kappa}{s'_k -\rho}\Big)
    = -\sideset{}{^{\prime}}\sum_{\beta\geq\frac{1}{2}}(F(s_k, \rho)-\kappa F(s'_k, \rho)) ,
\end{align*}
where 
$
    \sum'_{\beta\geq\frac{1}{2}}=\frac{1}{2}\sum_{\beta =\frac{1}{2}}+\sum_{\frac{1}{2}<\beta\leq 1}.
$ 
Using \eqref{2.21} for $k=1$, \eqref{2.23} for $k = 0,2,3, \ldots, n$, and this relation, we have
\begin{align*}
    S_1
    &= - \sum_{k=0}^{n} a_k\sum_{\varrho\in Z_L}\Re\Big(\frac{1}{s_k-\varrho} - \frac{\kappa}{{s'_k}-\varrho}\Big) \\
    &\leq - \frac{a_1}{\sigma - \beta_0} - \frac{a_0(\sigma - \beta_0)}{(\sigma - \beta_0)^2 +  \gamma_0^2} -  \sum_{k = 2}^{n}  \frac{a_k(\sigma - \beta_0)}{(\sigma - \beta_0)^2 + (k-1)^2 \gamma_0^2}+ \alpha_1 a_0 + \alpha_1 \sum_{k = 2}^{n}a_k.
\end{align*}
Secondly, we apply Lemma \ref{Lemma2.4} (with \eqref{2.26} for $k=0$ and \eqref{2.27} otherwise) to obtain
\begin{align*}
    S_3 
    &= \sum_{k=0}^{n} a_k \left(F(s_k,1) -\kappa F({s'_k},1)\right) \\
    &\leq \frac{a_0}{\sigma - 1} + a_0 \alpha_2 + \sum_{k=1}^{n} \frac{a_k(\sigma - 1)}{(\sigma - 1)^2 + k^2\gamma_0^2} + \alpha_3 \sum_{k=1}^{n} a_k . 
\end{align*}
Thirdly, it follows from Lemma \ref{Lemma 2.6} (with \eqref{2.32} for $k=0$ and \eqref{2.33} otherwise) that
\begin{align*}
    S_4 
    = \sum_{k=0}^{n} a_k \Re \Big(\frac{\gamma_L'(s_k)}{\gamma_L(s_k)} - \kappa \frac{\gamma_L'({s'_k})}{\gamma_L({s'_k})} \Big) 
    \leq \sum_{k=0}^{n} a_k \Big(d(k) - \frac{1-\kappa}{2} \log{\pi}  \Big) n_L .
\end{align*}
Our computations for $d(k)$ confirm that the coefficient of $n_L$ is negative, so we have
\begin{equation}\label{S4-upper-bd}
    S_4 \leq \sum_{k = 0}^{n} a_k \Big( d(k) - \frac{1-\kappa}{2} \log{\pi}  \Big) n_0 .
\end{equation}
Therefore, \eqref{eqn:imp} and the above inequalities yield
\begin{equation}\label{3.3}
\begin{split}
    0 &\leq \frac{a_0}{\sigma - 1} - \frac{a_1}{\sigma - \beta_0} + \frac{a_1(\sigma-1)}{(\sigma-1)^2 + \gamma_0^2} - \frac{a_0(\sigma-\beta_0)}{(\sigma-\beta_0)^2 + \gamma_0^2} + \frac{1-\kappa}{2} \mathscr{L}\sum_{k=0}^{n}a_k \\ 
    &\qquad + \sum_{k=2}^{n} a_k \Big(\frac{\sigma-1}{(\sigma-1)^2 + k^2\gamma_0^2} - \frac{\sigma-\beta_0}{(\sigma-\beta_0)^2 + (k-1)^2\gamma_0^2}\Big) \\
    &\qquad +\alpha_1 a_0 + \alpha_1 \sum_{k=2}^{n} a_k + \alpha_2a_0 + \alpha_3 \sum_{k=1}^{n} a_k + \sum_{k =0}^{n} a_k \Big( d(k) - \frac{1-\kappa}{2} \log{\pi}  \Big) n_0 .
\end{split}
\end{equation}
Using these observations, we prove the following lemma.

\begin{lemma}\label{lem:hotel_room}
Suppose that $\gamma_0 \in ( \frac{d_2}{\mathscr{L}}, 1 )$. Let $\mathcal{U}(x)$ be the unit step function defined by
\begin{equation}\label{unit-step-func}
\mathcal{U}(x) = \begin{cases} 
    1  & \text{ if } x \geq 0, \\
    0 & \text{ otherwise,}
\end{cases}
\end{equation}
and set
\begin{equation}\label{A_0case2}
    A_{n_0} = \alpha_1 a_0 + \alpha_1 \sum_{k=2}^{n} a_k + \alpha_2a_0 + \alpha_3 \sum_{k=1}^{n} a_k  + \sum_{k = 0}^{n} a_k  \Big(- \frac{1-\kappa}{2} \log \pi + d(k)  \Big) n_0.
\end{equation}
We have
\begin{equation*}
    0 \leq 
    \mathscr{L}\Big(\frac{a_0}{r} - \frac{a_1}{r+c} + \frac{a_1r}{r^2 + d_2^2} - \frac{a_0(r+c)}{(r+c)^2 + d_2^2} +  \frac{1-\kappa}{2} \sum_{k=0}^{n} a_k +  \frac{\mathcal{U}(A_{n_0}) A_{n_0}}{\mathscr{L}}  \Big).
\end{equation*}
\end{lemma}

\begin{proof}
The terms in the second row of \eqref{3.3} may be discarded, because
\begin{align*}
    \frac{\sigma-1}{(\sigma-1)^2 + k^2\gamma_0^2} - \frac{\sigma - \beta_0}{(\sigma-\beta_0)^2 +(k-1)^2\gamma_0^2} \leq 0
\end{align*}
for $k \in\{2,3,\ldots,n\}$; this observation extends \cite[Eqn.~(3.4)]{kadiri2012explicit}. Furthermore, \cite[Lem.~3.1(iii)]{kadiri2012explicit} tells us that if
\begin{equation*}
    a = \sigma-1 = \frac{r}{\mathscr{L}}, \quad
    b = \sigma - \beta_0 = \frac{r+c}{\mathscr{L}}, \quad
    q =  \frac{a_1}{a_0} , \quad\text{and}\quad
    x = \gamma_0 ,
\end{equation*}
then
\begin{align}\label{3.5}
\begin{split}
     \frac{a_1(\sigma-1)}{(\sigma-1)^2 + \gamma_0^2} - \frac{a_0(\sigma-\beta_0)}{(\sigma-\beta_0)^2 + \gamma_0^2} \leq a_0 f_3(a,b,q;x) ,
\end{split}
\end{align}
where $f_3(a,b,q;x)$ was defined on page 138 of \cite{kadiri2012explicit}. 
Moreover, $f_3(a,b,q;x)$ decreases as $x$ grows in the interval $(\frac{d_2}{\mathscr{L}},1)$, so
\begin{align*}
    f_3(a,b,q;x) \leq f_3\big(a,b,q;\frac{d_2}{\mathscr{L}}\big)
    = \Big(q\frac{r}{r^2 + d_2^2} - \frac{r+c}{(r+c)^2 + d_2^2}\Big) \mathscr{L}. 
\end{align*} 
Finally, the RHS of \eqref{3.3} is positive for $\gamma_0 \in ( \frac{d_2}{\mathscr{L}}, 1 )$. The result follows from this observation, \eqref{3.3}, and \eqref{3.5}.
\end{proof}

It follows from Lemma \ref{lem:hotel_room} and $\mathscr{L} \geq \mathscr{L}_0 > 0$ that
\begin{equation}\label{eqcase2}
    0 \leq  \frac{a_0}{r} - \frac{a_1}{r+c}  + \frac{a_1r}{r^2 + d_2^2}  - \frac{a_0(r+c)}{(r+c)^2 + d_2^2} + M_1 ,
\end{equation}
where 
\begin{equation*}
    M_1 = \frac{1-\kappa}{2} \sum_{k=0}^{n} a_k +  \frac{\mathcal{U}(A_{n_0}) A_{n_0}}{\mathscr{L}_0}.
\end{equation*} 
Therefore, \eqref{eqcase2} holds if and only if
\begin{equation*}
    \frac{a_1}{r+c} + \frac{a_0(r+c)}{(r+c)^2 + d_2^2} \leq \frac{a_0}{r} + \frac{a_1 r}{r^2 + d_2^2} + M_1,
\end{equation*}
which is equivalent to 
\begin{equation*}
    a_1 + a_0\Big(1 + \Big(\frac{d_2}{r+c}\Big)^2\Big)^{-1} \leq (r+c)\Big(\frac{a_0}{r} + \frac{a_1 r}{r^2 + d_2^2} + M_1\Big) .
\end{equation*}
It then follows that  
\begin{align*}
    c 
    &\geq \frac{a_1 +a_0\Big(1 + \Big(\frac{d_2}{r+c}\Big)^2\Big)^{-1} - \left(\frac{a_0}{r} + \frac{a_1 r}{r^2 + d_2^2} + M_1\right) r}{\frac{a_0}{r} + \frac{a_1 r}{r^2 + d_2^2} + M_1}, 
\end{align*}
where
$$ a_0\Big(1 + \Big(\frac{d_2}{r+c}\Big)^2\Big)^{-1}
>  a_0\bigg(1 + \bigg(\frac{d_2}{r+\frac{(a_1 - a_0) r}{a_0}}\bigg)^2\bigg)^{-1}
= a_0\Big(1 + \Big(\frac{a_0 d_2}{a_1 r}\Big)^2\Big)^{-1}.
$$
Hence, as $(1 - \beta_0)\mathscr{L} = c$, we have $ \beta_0 
    = 1 -\frac{c}{\mathscr{L}}$ and thus
\begin{align}
    \beta_0 
    &<  1 - \frac{1}{\mathscr{L}}\cdot r\cdot\frac{a_1 +
    a_0\Big(1 + \Big(\frac{a_0 d_2}{a_1 r}\Big)^2\Big)^{-1} - r \left(\frac{a_0}{r} + \frac{a_1 r}{r^2 + d_2^2} + M_1\right)}{a_0 + \frac{a_1}{1 + (d_2/r)^2} + M_1 r} 
    =: 1 - \frac{\tau_1(r,d_2)}{\mathscr{L}} . \label{eqn:for_comps_case2}
\end{align}

Table \ref{tab:Case2_optimised} presents choices of $r$, $d_2$, and one of the polynomials $p_n$ of degree $n\in\{8,16,40,46\}$ that were discussed in Section \ref{ssec:polynomial_choice} such that $\tau_1(r,d_2)^{-1}$ is minimised. To obtain these values, we used straightforward optimisation methods (up to five decimal places). Note that Tollis' verification work \cite{Tollis}, described in the introduction, means that we could safely assert 
\begin{equation}\label{eqn:TollisEtcLwrBound}
    d_L \geq 
    \begin{cases}
        400\,001 & \text{if $n_L = 2$;} \\
        240 & \text{if $n_L = 3$;} \\
        321 & \text{if $n_L = 4$;} \\
        d_{\text{min}}(n_L) & \text{if $n_L > 4$,} 
    \end{cases}
\end{equation}
in these computations. Recall that values for $d_{\text{min}}(n_L)$ are presented in Table \ref{Table-imp}. In the end, we expect the computations from this case to have the worst numerical outcome from Cases 2-4, so we fix the choices of $d_2$ associated to each $n_L$ from this case moving forwards. 

\begin{table}[]
    \centering
    \begin{tabular}{c|cccc}
        $n_L$ & polynomial & $r$ & $d_2$ & $\tau_1(r,d_2)^{-1}$ \\
        \hline
        $2$ & $p_{40}$ & $0.201$ & $0.70445$ & $16.018313$ \\
        $3$ & $p_{16}$ & $0.165$ & $0.56927$ & $19.552923$ \\
        $4$ & $p_{8}$ & $0.195$ & $0.68000$ & $16.722068$ \\
        $5$ & $p_{8}$ & $0.240$ & $0.83500$ & $13.712344$ \\
        $6$ & $p_{8}$ & $0.275$ & $0.91803$ & $11.781799$ \\
        $\geq 7$ & $p_{46}$ & $0.280$ & $0.98024$ & $11.519094$ 
    \end{tabular}
    \caption{Optimised parameter choices and computations for $\tau_1$.}
    \label{tab:Case2_optimised}
\end{table}


\subsection{Case 3}\label{case3}

Let $r$ and $c$ be taken as in \eqref{def-r&c} and suppose that
$
    0 < \frac{a_0}{a_1 - a_0}c < r < 1 .
$ 
Similar to Case $2$, using \eqref{2.21} for $k=1$, and \eqref{2.24} otherwise, we get
\begin{align*}
    S_1 
    &\leq - \frac{a_1}{\sigma - \beta_0} + 2\alpha_1 a_0  + 2\alpha_1 \sum_{k=2}^{n} a_k 
    - \frac{2a_0(\sigma - \beta_0)}{(\sigma - \beta_0)^2 + \gamma_0^2} \\
    &\qquad 
    - \sum_{k=2}^{n} a_k \Big( \frac{\sigma - \beta_0}{(\sigma-\beta_0)^2 + (k-1)^2 \gamma_0^2} + \frac{\sigma - \beta_0}{(\sigma-\beta_0)^2 + (k+1)^2 \gamma_0^2}\Big).
\end{align*}
Next, we apply Lemma \ref{Lemma2.4} (with \eqref{2.26} for $k=0$ and \eqref{2.27} otherwise) to deduce
\begin{equation*}
    S_3 \leq \frac{a_0}{\sigma-1} + a_0 \alpha_2 + \sum_{k = 1}^{n} \frac{a_k(\sigma - 1)}{(\sigma - 1)^2 + k^2\gamma_0^2} + \alpha_3 \sum_{k=1}^{n}a_k.
\end{equation*}
In addition, as argued in the previous section, applying Lemma \ref{Lemma 2.6} with \eqref{2.32} for $k=0$ and \eqref{2.33} otherwise, together with  our computations for $d(k)$ confirming that the coefficient of $n_L$ is negative, we again have the upper bound \eqref{S4-upper-bd} for $S_4$.

Combining all the above inequalities with \eqref{eqn:imp}, we deduce
\begin{align}
 \begin{split}
    0 \leq &\frac{a_0 }{\sigma-1} - \frac{a_1}{\sigma-\beta_0} + \frac{a_1(\sigma-1)}{(\sigma-1)^2 + \gamma_0^2} - \frac{2a_0(\sigma-\beta_0)}{(\sigma-\beta_0)^2 + \gamma_0^2}  + \frac{1-\kappa}{2} \Big(\sum_{k=0}^{n}a_k \Big) \mathscr{L} \\
    &+ \sum_{k=2}^{n} a_k \!\Big( \frac{\sigma - 1}{(\sigma-1)^2 + k^2\gamma_0^2} - \frac{\sigma - \beta_0}{(\sigma-\beta_0)^2 + (k-1)^2\gamma_0^2} - \frac{\sigma - \beta_0}{(\sigma-\beta_0)^2 + (k+1)^2\gamma_0^2}\Big)  \\
    & +2\alpha_1 a_0+ 2 \alpha_1 \sum_{k=2}^{n} a_k + \alpha_2 a_0 + \alpha_3 \sum_{k=1}^{n} a_k \label{3.8} 
    + \sum_{k = 1}^n a_k \Big(d(k) -\frac{1-\kappa}{2} \log{\pi} \Big) n_0 . 
    \end{split}
\end{align}
Here, $k \geq 2$ and $\gamma_0 \in (\frac{d_1}{\mathscr{L}},\frac{d_2}{\mathscr{L}})$, so we have
\begin{equation}\label{6.5}
\begin{split}
    &\frac{\sigma-1}{(\sigma-1)^2 + k^2\gamma_0^2} - \frac{\sigma-\beta_0}{(\sigma-\beta_0)^2 + (k-1)^2\gamma_0^2} - \frac{\sigma-\beta_0}{(\sigma-\beta_0)^2 + (k+1)^2\gamma_0^2} \\
    &\qquad\qquad\leq \Big( \frac{r}{r^2 + k^2d_1^2} - \frac{r+c}{(r+c)^2 + (k-1)^2d_2^2 } - \frac{r+c}{(r+c)^2 + (k+1)^2d_2^2 }\Big) \mathscr{L}.   
\end{split}
\end{equation}
\begin{remark}
    Observe that \eqref{3.8} closely mirrors the inequality given in \cite[Pg. 143, Eq.~(3.8)]{kadiri2012explicit}, except for the fifth term on the right of \cite[Eq.~(3.8)]{kadiri2012explicit}, which appears to be an error. In addition, \eqref{6.5} corresponds to the inequality stated below \cite[Eq.~(3.8)]{kadiri2012explicit},
  but we correct a typographical error there by placing $\mathscr{L}$ in the numerator (instead of having $\frac{1}{\mathscr{L}}$).
\end{remark}
Noticing that $r$ and $c$ satisfy the conditions we imposed earlier, similar to the derivation of \eqref{3.5}, as
$
    \frac{a_1(\sigma-1)}{(\sigma - 1)^2 + \gamma_0^2} - \frac{a_0 (\sigma - \beta_0)}{(\sigma-\beta_0)^2 + \gamma_0^2}
$ 
decreases with $\gamma_0 \in (\frac{d_1}{\mathscr{L}},\frac{d_2}{\mathscr{L}})$, we deduce
\begin{equation}\label{6.6}
    \frac{a_1(\sigma-1)}{(\sigma-1)^2 + \gamma_0^2} - \frac{a_0(\sigma-\beta_0)}{(\sigma-\beta_0)^2 + \gamma_0^2}
    \leq \Big( \frac{a_1r}{r^2 + d_1^2} - \frac{a_0(r+c)}{(r+c)^2 + d_1^2}\Big) \mathscr{L} .
\end{equation}
It also follows from $\gamma_0 \in (\frac{d_1}{\mathscr{L}},\frac{d_2}{\mathscr{L}})$ that
\begin{equation}\label{6.7}
    -\frac{a_0(\sigma-\beta_0)}{(\sigma-\beta_0)^2 + \gamma_0^2} 
    \leq - \Big(\frac{a_0(r+c)}{(r+c)^2 + d_2^2}
    \Big)\mathscr{L}.
\end{equation}


Now, the equation in \eqref{3.8} is positive for $\gamma_0 \in ( \frac{d_1}{\mathscr{L}}, \frac{d_2}{\mathscr{L}} )$, so we can combine \eqref{def-r&c}, \eqref{3.8}, \eqref{6.5}, \eqref{6.6}, and \eqref{6.7} to derive
\begin{align*}
    0 
    &\leq 
    \frac{a_0}{r} - \frac{a_1}{r+c} + \frac{a_1r}{r^2 + d_1^2} - \frac{a_0(r+c)}{(r+c)^2 + d_1^2} - \frac{a_0(r+c)}{(r+c)^2 + d_2^2}  + \frac{1-\kappa}{2} \sum\limits_{k=0}^{n} a_k \\
    &\qquad+ \sum\limits_{k=2}^{n} a_k \Big( \frac{r}{r^2 + k^2d_1^2} - \frac{r+c}{(r+c)^2 + (k-1)^2 d_2^2} - \frac{r+c}{(r+c)^2 + (k+1)^2 d_2^2}  \Big) +\frac{\mathcal{U}(B_{n_0}) B_{n_0}}{\mathscr{L}} ,
\end{align*}
where $\mathcal{U}(x)$ is the unit step function defined as in \eqref{unit-step-func} and
\begin{equation}\label{eqn:Bn}
    B_{n_0} = 2\alpha_1 a_0+ 2 \alpha_1 \sum_{k=2}^{n} a_k + \alpha_2 a_0 + \alpha_3 \sum_{k=1}^{n} a_k \\
    + \sum_{k = 1}^n a_k \Big(d(k) -\frac{1-\kappa}{2} \log{\pi}  \Big) n_0.
\end{equation} 
It follows from $\mathscr{L} \geq \mathscr{L}_0 > 0$ that
\begin{align}\label{eqcase3}
\begin{split}
    0 & \leq \frac{a_0}{r} - \frac{a_1}{r+c} + \frac{a_1r}{r^2 + d_1^2} - \frac{a_0(r+c)}{(r+c)^2 + d_1^2} - \frac{a_0(r+c)}{(r+c)^2 + d_2^2} \\
     &\quad  + \sum\limits_{k=2}^{n} a_k \Big( \frac{r}{r^2 + k^2d_1^2} - \frac{r+c}{(r+c)^2 + (k-1)^2 d_2^2} - \frac{r+c}{(r+c)^2 + (k+1)^2 d_2^2} \Big) + M_2,
\end{split} 
\end{align}
where 
\begin{equation*}
    M_2 = \frac{1-\kappa}{2} \sum\limits_{k=0}^{n} a_k + \frac{\mathcal{U}(B_{n_0}) B_{n_0}}{\mathscr{L}_0} .
\end{equation*}

By the condition $r+c \leq \frac{a_1}{a_0} r$, we must have
\begin{equation}\label{required}
    \frac{1}{1 + \frac{z_0}{(r+c)^2}} \leq \frac{1}{1+ \frac{z_0 a_0^2}{a_1^2 r^2}}.
\end{equation}
Therefore, the inequality \eqref{eqcase3} holds if the following is true:
\begin{align*}
    &a_1 + \frac{a_0}{1+ \frac{d_1^2 a_0^2}{r^2 a_1^2}} + \frac{a_0}{1+ \frac{d_2^2 a_0^2}{r^2 a_1^2}} + \sum\limits_{k=2}^{n} a_k Q(k,d_2,r;a_0,a_1) 
   \leq (r+c) \bigg(r \sum\limits_{k=0}^{n} \frac{a_k}{r^2 + k^2d_1^2} + M_2\bigg), 
\end{align*}
where
\begin{equation}\label{Qkda}
    Q(k,d,r;a_0,a_1) = 
\frac{1}{1 + \frac{(k-1)^2 d^2 a_0^2}{a_1^2 r^2}} 
+ \frac{1}{1 + \frac{(k+1)^2 d^2 a_0^2}{a_1^2 r^2}}.
\end{equation}
Clearly, the preceding inequality is satisfied if and only if 
\begin{equation*}
    c \ge
    \frac{a_1 + \frac{a_0}{1+ \frac{d_1^2 a_0^2}{r^2 a_1^2}} + \frac{a_0}{1+ \frac{d_2^2 a_0^2}{r^2 a_1^2}} + \sum\limits_{k=2}^{n} a_k Q(k,d_2,r;a_0,a_1) 
    - r \left( r \sum_{k=0}^n \frac{a_k}{r^2+k^2 d_1^2} + M_2 \right)}{r \sum_{k=0}^n \frac{a_k}{r^2+k^2 d_1^2} + M_2}.
\end{equation*}
Hence, as $1 - \beta_0 = \frac{c}{\mathscr{L}}$, it follows that
\begin{align}
    \beta_0 
    &\le
    1 - \frac{a_1 + \frac{a_0}{1+ \frac{d_1^2 a_0^2}{r^2 a_1^2}} + \frac{a_0}{1+ \frac{d_2^2 a_0^2}{r^2 a_1^2}} + \sum\limits_{k=2}^{n} a_k Q(k,d_2,r;a_0,a_1)  
    - r \left( r \sum_{k=0}^n \frac{a_k}{r^2+k^2 d_1^2} + M_2 \right)}{\left(r \sum_{k=0}^n \frac{a_k}{r^2+k^2 d_1^2} + M_2\right) \mathscr{L}} \nonumber\\
    &=: 1 - \frac{\tau_2(r,d_1,d_2)}{\mathscr{L}} . \label{eqn:for_comps_case3}
\end{align}
We defer the computations for this case until Section \ref{ssec:proof_small_t}, as our eventual choice of $d_1$ will also have an effect on the outcome in Case 4.


\subsection{Case 4}\label{case4}

We will utilise a different argument in this case. To begin, we note that
\begin{align}\label{fL-expansion}
 \begin{split}
    &f_L(\sigma, 0)\\
    &= - \sideset{}{^{\prime}}\sum_{\substack{\varrho\in Z_{L} \\ \beta \geq 1/2}} \left(F(\sigma,\varrho) - \kappa F(\sigma_1,\varrho)\right)+\frac{1-\kappa}{2}\mathscr{L}
    + F(\sigma,1) - \kappa F(\sigma_1,1) +\Re\Big(\frac{\gamma^{\prime}_{L}}{\gamma_{L}}(\sigma)-\kappa\frac{\gamma^{\prime}_{L}}{\gamma_{L}}(\sigma_{1})\Big) .
    \end{split}
\end{align}
It follows from Lemmas \ref{Lemma 2.3}, \ref{Lemma2.4}, and  \ref{Lemma 2.6} that
\begin{equation*}
\begin{split}
    f_L(\sigma, 0) 
    \leq 2\Big(\alpha_1 - \frac{\sigma- \beta_0}{(\sigma - \beta_0)^2 + \gamma_0^2}\Big) 
     + \frac{1-\kappa}{2}\mathscr{L} +  \frac{1}{\sigma -1} + \alpha_2 +  \Big(d(0) - \frac{(1-\kappa) \log \pi}{2} \Big) n_L.
\end{split}
\end{equation*}
Recall that $\sigma - 1 = \frac{r}{\mathscr{L}}$, $1 - \beta_0 = \frac{c}{\mathscr{L}}$, and note that the coefficient of $n_L$ in the above equation is negative. Therefore, $f_L(\sigma, 0) \geq 0$, $n_L \geq n_0$, and $\gamma_0 \leq \frac{d_1}{\mathscr{L}}$ imply
\begin{equation}\label{eqn:case4_convenient}
    0 \leq 
    \mathscr{L} \left( \frac{1}{r} - 2 \frac{r+c}{(r+c)^2 + d_1^2} + \frac{1 - \kappa}{2} + \frac{\mathcal{U}(C_{n_0}) C_{n_0}}{\mathscr{L}} \right) , 
\end{equation}
where $\mathcal{U}(x)$ is defined as in \eqref{unit-step-func} and
\begin{equation}\label{C0case4}
    C_{n_0} = 2\alpha_1 + \alpha_2 + n_0 \Big( d(0) - \frac{(1-\kappa) \log \pi}{2} \Big).
\end{equation}

Since $\mathscr{L} \geq \mathscr{L}_0 > 0$, the preceding discussion tells us
\begin{equation}\label{eqcase4}
    0 \leq \frac{1}{r} - 2 \frac{r+c}{(r+c)^2 + d_1^2} + M_3,
\end{equation}
in which
\begin{equation}\label{M}
    M_3 = \frac{1 - \kappa}{2} + \frac{\mathcal{U}(C_{n_0}) C_{n_0}}{\mathscr{L}_0}.
\end{equation}
Note that \eqref{eqcase4} is equivalent to 
\begin{equation*}
    c^2 + \Big( \frac{2 M_3 r^2}{1+M_3 r}\Big) c + \Big(d_1^2 + \frac{M_3 r^3 - r^2}{1+M_3 r}\Big) \geq 0.
\end{equation*}
As the roots of the above quadratic equation occur at
\begin{equation*}
    c = \frac{- M_3r^2 \pm \sqrt{r^2 - d_1^2(1+M_3r)^2}}{1 + M_3r},
\end{equation*}
the last inequality above holds when 
\begin{equation*}
    c \geq \frac{\sqrt{r^2 - d_1^2(1+M_3r)^2} - M_3r^2}{1 + M_3r},
\end{equation*}
which yields
\begin{equation}\label{eqn:for_comps_case4}
    \beta_0 \leq 1 - \bigg(\frac{\sqrt{r^2 - d_1^2(1+M_3r)^2} - M_3r^2}{1 + M_3r}\bigg) \frac{1}{\mathscr{L}}
    =: 1 - \frac{\tau_3(r,d_1)}{\mathscr{L}} .
\end{equation}


\subsection{Proof of Theorem \ref{thm:for|t|<1}}\label{ssec:proof_small_t}

Recall that the coefficients of $\mathscr{L}^{-1}$ in \eqref{eqn:for_comps_case2}, \eqref{eqn:for_comps_case3}, and \eqref{eqn:for_comps_case4} respectively, which we aim to maximise, are
\begin{align*}
    \tau_1(r,d_2) &= \frac{a_1 + a_0\Big(1 + \Big(\frac{a_0 d_2}{a_1 r}\Big)^2\Big)^{-1} - r \left(\frac{a_0}{r} + \frac{a_1 r}{r^2 + d_2^2} + M_1\right)}{\frac{a_0}{r} + \frac{a_1 / r}{1 + (d_2/r)^2} + M_1} , \\
    \tau_2(r,d_1,d_2) &= \frac{a_1 + \frac{a_0}{1+ \frac{d_1^2 a_0^2}{r^2 a_1^2}} + \frac{a_0}{1+ \frac{d_2^2 a_0^2}{r^2 a_1^2}} + \sum\limits_{k=2}^{n} a_k Q(k,d_2,r;a_0,a_1)  
    - r \left( r \sum_{k=0}^n \frac{a_k}{r^2+k^2 d_1^2} + M_2 \right)}{r \sum_{k=0}^n \frac{a_k}{r^2+k^2 d_1^2} + M_2} , \\
    \tau_3(r,d_1) &= \frac{\sqrt{r^2 - d_1^2(1+M_3r)^2} - M_3r^2}{1 + M_3r} .
\end{align*}
Equivalently, we are seeking $r$, $d_1$, $d_2$ such that the largest value of $\tau_1(r,d_2)^{-1}$, $\tau_2(r,d_1,d_2)^{-1}$, and $\tau_3(r,d_1)^{-1}$ is minimised. We have already presented optimised computations for $\tau_1(r,d_2)^{-1}$ in Table \ref{tab:Case2_optimised}, which we expect to be the worst case, so all that remains is to fix the values of $d_2$ presented in Table \ref{tab:Case2_optimised}, then choose $r$ and $d_1$ such that $\max\{\tau_2(r,d_1,d_2)^{-1}, \tau_3(r,d_1)^{-1}\}$ is minimal. For these computations, we choose the polynomials based on numerical experiments, and we implement the lower bound from \eqref{eqn:TollisEtcLwrBound} for $d_L$. Furthermore, we note that our choice of $r$ does not need to be the same when computing $\tau_1(r,d_2)^{-1}$, $\tau_2(r,d_1,d_2)^{-1}$, and $\tau_3(r,d_1)^{-1}$; the only parameters that need to be consistent across each case are $d_1$ and $d_2$. Moreover, to ensure that $\tau_3(r,d_1)$ is real, $d_1$ must be chosen so that
$
    r^2 \geq d_1^2(1+M_3r)^2 .
$

Finally, fixing the choices of $d_2$ presented in Table \ref{tab:Case2_optimised} and using the polynomial $p_{46}$, we have determined that the choices of $r$ and $d_1$ presented in Table \ref{tab:optimisations_Sourabh} give the best outcomes (up to four decimal places). Theorem \ref{thm:for|t|<1} mainly follows from these resulting computations; the only remaining part of the proof is to deal with the real zeros that will be handled in Subsection \ref{ssec:ez_upper_bound}.

\begin{remark}
The computations in Table \ref{tab:optimisations_Sourabh} also demonstrate how much improvement would be available in Theorem \ref{thm:for|t|<1} if future advancements enabled us to reduce the size of $\tau_1(r,d_2)^{-1}$. 
\end{remark}

\begin{table}[]
    \centering
    \begin{tabular}{c|ccccccc}
        $n_L$ & polynomial & $r_{A}$ & $r_{B}$ & $d_1$ & $d_2$ & $\tau_2(r_{A},d_1,d_2)^{-1}$ & $\tau_3(r_{B},d_1)^{-1}$ \\
        \hline
        2 & $p_{46}$ & 0.4827 & 1.2294 & 0.5353 & 0.70445 & 6.099311 & 6.099299 \\
        3 & $p_{16}$ & 0.3348 & 0.8220 & 0.4140 & 0.56927 &  9.968097 & 9.963504 \\
        4 & $p_{8}$ & 0.3590 & 0.8656 & 0.4768 & 0.68000 & 9.890467 & 9.890536 \\
        5 & $p_{8}$ & 0.4522 & 1.1072 & 0.6004 & 0.83500 & 7.981275 & 7.976893 \\
        6 & $p_{8}$ & 0.5545 & 1.3776 & 0.7122 & 0.91803 & 6.029385 & 6.026121 \\
        7 & $p_{46}$ & 0.6533 & 1.8187 & 0.8173 & 0.98024 & 4.595864 & 4.595872 \\
        $\geq 8$ & $p_{46}$ & 0.7190 & 1.9524 & 0.8492 & 0.98024 & 3.884018 & 3.884011 \\
    \end{tabular}
    \caption{Optimised parameter choices and computations for $\tau_2$ and $\tau_3$.}
    \label{tab:optimisations_Sourabh}
\end{table}

\subsection{Proof of Theorem \ref{thm:exp-zero-ZFR}}\label{section::exceptional}

Suppose that an exceptional zero, denoted $\beta_1$, exists such that
$
    \beta_1 \geq 1 - \frac{\nu}{\mathscr{L}} ,
$ where $\nu > 0$. Note that if $k=0$ and $\varrho = \beta + i\gamma \in Z_L$ such that $\gamma = 0$, then 
\begin{align*}
    F(s_k,\varrho)-\kappa F(s'_k,\varrho) 
    &= \frac{\sigma-\beta}{(\sigma-\beta)^2 + (kt-\gamma)^2} + \frac{\sigma-1+\beta}{(\sigma-1+\beta)^2 + (kt-\gamma)^2} \\
    &\qquad - \kappa \Big(\frac{\sigma_1-\beta}{(\sigma_1-\beta)^2 + (kt-\gamma)^2} + \frac{\sigma_1-1+\beta}{(\sigma_1-1+\beta)^2 + (kt-\gamma)^2}\Big) \\ 
    &= \frac{1}{\sigma-\beta} + \frac{1}{\sigma-1+\beta} - \kappa \Big(\frac{1}{\sigma_1-\beta} + \frac{1}{\sigma_1-1+\beta}\Big) .
\end{align*}
Hence, it follows from \cite[Eqn.~(2.16)]{kadiri2012explicit} that if $k=0$ and $\varrho = \beta + i\gamma \in Z_L$ such that $\gamma = 0$, then 
\begin{equation}\label{eqn:v_important}
    F(s_k,\varrho)-\kappa F(s'_k,\varrho) \geq \frac{1}{\sigma-\beta} .
\end{equation}
Repeating the arguments in Sections \ref{case2}-\ref{case4}, while using the bound \eqref{eqn:v_important} at the opportune moment to account for the extra negative contribution given by the exceptional zero $\beta_1$, we see that an isolated zero $\varrho_0 = \beta_0 + i\gamma_0 \in Z_L$ such that $0 < |\gamma_0| \leq 1$ will satisfy
\begin{align*}
    \beta_0 
    &< 1 - 
    \begin{cases}
        \frac{\widehat{\tau_1}(r,d_2,\nu)}{\mathscr{L}} &\text{if $\frac{d_2}{\mathscr{L}} < |\gamma_0| \leq 1$;}\\
        \frac{\widehat{\tau_2}(r,d_1,d_2,\nu)}{\mathscr{L}} &\text{if $\frac{d_1}{\mathscr{L}} < |\gamma_0| \leq \frac{d_2}{\mathscr{L}}$;}\\
        \frac{\widehat{\tau_3}(r,d_1,\nu)}{\mathscr{L}} &\text{if $0 < |\gamma_0| \leq \frac{d_1}{\mathscr{L}}$,}
    \end{cases}
\end{align*}
in which
\begin{align*}
    \widehat{\tau_1}(r,d_2,\nu) &= \frac{a_1 + a_0\Big(1 + \Big(\frac{a_0 d_2}{a_1 r}\Big)^2\Big)^{-1} - r \left(\frac{a_0}{r} - \frac{a_0}{r+\nu} + \frac{a_1 r}{r^2 + d_2^2} + M_1\right)}{\frac{a_0}{r} - \frac{a_0}{r+\nu} + \frac{a_1 / r}{1 + (d_2/r)^2} + M_1} , \\
    \widehat{\tau_2}(r,d_1,d_2,\nu) &= \frac{a_1 + \frac{a_0}{1+ \frac{d_1^2 a_0^2}{r^2 a_1^2}} + \frac{a_0}{1+ \frac{d_2^2 a_0^2}{r^2 a_1^2}} + \sum\limits_{k=2}^{n} a_k Q(k,d_2,r;a_0,a_1)}{r \sum_{k=0}^n \frac{a_k}{r^2+k^2 d_1^2} - \frac{a_0}{r+\nu} + M_2} - r , \\
    \widehat{\tau_3}(r,d_1,\nu) & = \frac{- \left( M_3r^2 + r \left( \frac{\nu}{r+\nu} - 1 \right) \right) + \sqrt{r^2 - d_1^2 \left(M_3r + \frac{\nu}{r+\nu} \right)^2}}{M_3r + \frac{\nu}{r+\nu}}.
\end{align*}
Again, we need to choose the parameters $r$, $d_1$, and $d_2$ such that the largest value of $\widehat{\tau_1}(r,d_2)^{-1}$, $\widehat{\tau_2}(r,d_1,d_2)^{-1}$, and $\widehat{\tau_3}(r,d_1)^{-1}$ is minimised. 

For the optimisation carried out in Section \ref{ssec:proof_small_t}, we observed that the best $\tau_1(r,d_2)^{-1}$ was greater than the optimal $\max  \{ \tau_2(r,d_1,d_2)^{-1}, \tau_3(r,d_1)^{-1} \}$ values, when computing with the same $d_2$. Thus, we optimised $\tau_1(r,d_2)^{-1}$ first and used the $d_2$ obtained in the process to compute optimal $\max  \{ \tau_2(r,d_1,d_2)^{-1}, \tau_3(r,d_1)^{-1} \}$. 

Here, the situation is different. We notice that the $d_2$ obtained from the best $\widehat{\tau_1}(r,d_2)^{-1}$ produces values of $\max  \{ \widehat{\tau}_2(r,d_1,d_2,\nu)^{-1}, \widehat{\tau}_3(r,d_1^{-1},\nu)  \}$ which are greater than $\widehat{\tau_1}(r,d_2)^{-1}$. Therefore, we adopt a different strategy to address this case. We optimise the parameters $r$, $d_1$, and $d_2$ together such that 
$\max \{ \widehat{\tau_1}(r,d_2,\nu)^{-1},\widehat{\tau}_2(r,d_1,d_2,\nu)^{-1}, \widehat{\tau}_3(r,d_1,\nu)^{-1} \}$ 
is as small as possible. 
The case $\nu = 0.5$ is given in Table \ref{tab:optimisations_nu=0.5} and the case $\nu = 0.05$ is given in Table \ref{tab:optimisations_nu=0.05}. We remark that in all these computations, the polynomial $p_4$ produced the best values. Finally, we note that Theorem \ref{thm:exp-zero-ZFR} is a natural consequence of the resulting computations, except for the case of zeros on the real line, which is handled in Section \ref{ssec:ez_upper_bound_except}.

\begin{table}[]
    \centering
    \begin{small}
    \begin{tabular}{c|cccccccc}
        $n_L$ & $r_A$ & $r_{B}$ & $r_{C}$ & $d_1$ & $d_2$ & $\hat{\tau}_1(r_{A},d_2,0.5)^{-1}$& $\hat{\tau}_2(r_{B},d_1,d_2,0.5)^{-1}$ & $\hat{\tau}_3(r_{C},d_1,0.5)^{-1}$ \\
        \hline
        2 & 0.38 & 1.11 & 1.93 & 1.82 & 1.82 & 6.036555 & 6.011471 & 5.990305 \\
        3 & 0.30 & 0.47 & 0.82 & 0.95 & 1.22 & 8.253321 & 8.251776 & 5.056856 \\
        4 & 0.36 & 0.57 & 0.86 & 1.06 & 1.27 & 6.422941 & 6.415110 & 2.951337 \\
        5 &  0.45 & 0.72 & 1.10 & 1.44 & 1.56 &  4.668822 & 4.659528 & 2.192861  \\
        6 & 0.53 & 0.80 & 1.37  & 1.85 & 1.85 & 3.659933 & 3.659357 & 1.458647 \\
        7 & 0.47 & 0.79 & 1.81 & 2.40 & 2.50 & 3.640973 & 3.644519 & 1.309188 \\
        8 & 0.54 & 0.94 & 2.10 &  2.71 & 2.93 &  3.576987 & 3.571527 & 1.217929 \\
        9 & 0.54 & 1.50 & 2.54 & 2.93 & 2.93 & 3.576987 & 3.550708 & 2.090838 \\
        10 & 0.54 & 1.78 & 2.82 & 2.93 & 2.93 & 3.576987 & 3.562360 & 1.007469 \\
        11 & 0.54 & 2.00 & 3.31 & 2.93 & 2.93 & 3.576987 & 3.542913 & 0.9535690 \\
        $\geq$ 12 & 0.54 & 2.08 & 3.60 & 2.93 & 2.93 & 3.576987 & 3.545561 & 1.035471 \\
    \end{tabular}
    \end{small}
    \caption{Optimised parameter choices and computations for $\hat{\tau}_1, \hat{\tau}_2$ and $\hat{\tau}_3$ with $\nu = 0.5$ and $p_4$.}
    \label{tab:optimisations_nu=0.5}
\end{table}

\begin{table}[]
    \centering
    \begin{small}
    \begin{tabular}{c|cccccccc}
        $n_L$ & $r_A$ & $r_{B}$ & $r_{C}$ & $d_1$ & $d_2$ & $\hat{\tau}_1(r_{A},d_2,0.05)^{-1}$& $\hat{\tau}_2(r_{B},d_1,d_2,0.05)^{-1}$ & $\hat{\tau}_3(r_{C},d_1,0.05)^{-1}$ \\
        \hline
        2 & 0.54 & 0.58 & 1.93 & 2.14 & 2.14 & 2.576574 & 2.565449 & 1.135791 \\
        3 & 0.42 & 0.45 & 0.82 & 1.12 & 1.12 & 3.316475 &  3.311881 & 0.5795893 \\
        4 & 0.49 & 0.50 & 0.86 & 1.16 & 1.16 & 2.851720 & 2.847493 & 0.4563488 \\
        5 & 0.61 & 0.61 & 1.10 & 1.41 & 1.41 & 2.312646 & 2.310428 & 0.3490578 \\
        6 & 0.71 & 0.75 & 1.37 & 1.61 & 1.61 & 1.960456 & 1.954740 & 0.2828832 \\
        7 & 0.75 & 0.79 & 1.81 & 2.27 & 2.27 & 1.855168 & 1.852961 & 0.2713828 \\
        8 & 0.76 & 0.91 & 2.10 & 2.70 & 2.70 & 1.830414 & 1.826330 & 0.2799502 \\
        9 & 0.77 & 0.78 & 2.54 & 3.52 & 3.90 & 1.806471 & 1.804384 & 0.9316547 \\
        10 & 0.77 & 0.95 & 2.82 & 3.53 & 3.91 & 1.806301 & 1.802006 & 1.030697 \\
        11 & 0.77 & 1.13 & 3.23 & 3.55 & 3.91 & 1.806301 & 1.798172 & 1.801953 \\
        $\geq$ 12 & 0.77 & 1.20 & 3.23 & 3.55 & 3.91 & 1.806301 & 1.798172 & 1.801953 \\
    \end{tabular}
    \end{small}
    \caption{Optimised parameter choices and computations for $\hat{\tau}_1, \hat{\tau}_2$ and $\hat{\tau}_3$ with $\nu = 0.05$ and $p_4$.}
    \label{tab:optimisations_nu=0.05}
\end{table}


\section{Case 5: bounds for the real zero}\label{upper_bound}

Bring forward all of the notations and set-up from Section \ref{sec:setup}, in this section, we will complete the proof of Theorems \ref{thm:for|t|<1}-\ref{thm:onrealline-exceptional}. In particular, results from Subsection \ref{ssec:ez_upper_bound} complete proofs of Theorems \ref{thm:for|t|<1} and \ref{thm:onrealline}, and results from Subsection \ref{ssec:ez_upper_bound_except} complete proofs of Theorems \ref{thm:exp-zero-ZFR} and \ref{thm:onrealline-exceptional}.

\subsection{Case 5: real zeros with no assumption on the existence of an exceptional zero}\label{ssec:ez_upper_bound}

To begin, suppose that $\sigma - 1 = \frac{r}{\mathscr{L}}$, $1 - \beta_0 = \frac{c}{\mathscr{L}}$, and $\zeta_L(s)$ admits two real zeros $\beta_0$ and $\beta_1$ such that  $\beta_0 \leq \beta_1$. Similar to Subsection \ref{case4}, by \eqref{fL-expansion}, Lemmas \ref{Lemma2.4}-\ref{Lemma 2.6}, \eqref{eqn:v_important} with $\varrho = \beta_0$ and $\varrho = \beta_1$, and the fact that
$
    \frac{1}{\sigma - \beta_0} + \frac{1}{\sigma - \beta_1} \geq \frac{2}{\sigma - \beta_0},
$
we derive
\begin{equation}
\begin{split}
    f_L(\sigma, 0) 
    &\leq - \frac{2}{\sigma - \beta_0} + \frac{1-\kappa}{2}\mathscr{L} +  \frac{1}{\sigma -1} + \alpha_2 +  \Big(d(0) - \frac{(1-\kappa) \log \pi}{2} \Big) n_L.
\end{split}
\end{equation}
Therefore, $f_L(\sigma,0) \geq 0$ and $\mathscr{L} \geq \mathscr{L}_0 > 0$ imply
\begin{align}
    0 
    &\leq - \frac{2}{\sigma - \beta_0} + \mathscr{L} \Big(\frac{1-\kappa}{2} + \frac{1}{r}\Big) + \alpha_2 +  \Big( d(0) - \frac{(1-\kappa) \log \pi}{2} \Big) n_L \nonumber\\
    &= \mathscr{L} \Big(\frac{1-\kappa}{2} + \frac{1}{r} - \frac{2}{r + c}\Big) + \alpha_2 +  \Big( d(0) - \frac{(1-\kappa) \log \pi}{2} \Big) n_L \nonumber\\
    &\leq \mathscr{L} \Big(\frac{1-\kappa}{2} + \frac{1}{r} - \frac{2}{r + c} + \frac{\mathcal{U}(D_{n_L}) D_{n_L}}{\mathscr{L}_0}\Big) , \label{eqn:ub_important_2}
\end{align}
where $\mathcal{U}(x)$ was defined in \eqref{unit-step-func} and 
\begin{equation}\label{D0case5}
    D_{n_0} = \alpha_2 + n_0 \Big( d(0) - \frac{(1-\kappa) \log \pi}{2} \Big).
\end{equation}
Rearranging \eqref{eqn:ub_important_2}, we see that
\begin{equation}\label{eqn:ub_important1}
    c \geq \eta(n_L,r) := \frac{2}{\frac{1-\kappa}{2} + \frac{1}{r} + \frac{\mathcal{U}(D_{n_L}) D_{n_L}}{\mathscr{L}_0}} - r .
\end{equation}

For each $n_L\in \{2,3,4,5,6,7\}$, optimised choices for $r$ and the resulting values of $\eta(n_L,r)$ are presented in Table \ref{tab:optimisations_ez}. Similar to our earlier computations, we asserted $d_L \geq 400\,001$ when $n_L = 2$, $d_L \geq 240$ when $n_L = 3$, $d_L \geq 321$ when $n_L = 4$, and $d_L \geq d_{\text{min}}(n_L)$ when $n_L > 4$. However, we must be careful to choose $r$ such that $r < 0.15\mathscr{L}_0$, as most of our computations require $\sigma < 1.15$. Therefore, our algorithm to find the optimal $r$ follows. 

\begin{algorithmic}
\State $r \gets 0.00001$
\While{$\eta(n_L,r) < \eta(n_L,r + 0.00001)$}
    \If{$r+0.00001 \leq 0.15\mathscr{L}_0$}
    \State $r \gets r + 0.00001$
    \EndIf
\EndWhile
\end{algorithmic}

\noindent It is also clear that the lower bound in \eqref{eqn:ub_important1} grows with $n_L \geq 7$, so $n_L \geq 7$ implies $c \geq \eta(7,r)$. These computations complete the proof of Theorem \ref{thm:for|t|<1}.


\begin{table}[]
    \centering
    \begin{tabular}{ccc}
        $n_L$ & $r$ & $\eta(n_L,r)^{-1}$ \\
        \hline
        2 & 1.49859 & 1.61094 \\
        3 & 0.822093 & 1.93173 \\
        4 & 0.865713 & 1.88178 \\
        5 & 1.10750 & 1.69958 \\
        6 & 1.37771 & 1.61857 \\
        $\geq$ 7 & 1.49859 & 1.61094
    \end{tabular}
    \caption{Optimised parameter choices and computations for $\eta(n_L,r)$.}
    \label{tab:optimisations_ez}
\end{table}

\subsection{Case 5: real zeros with exceptional zero}\label{ssec:ez_upper_bound_except} 

Assume that $\sigma - 1 = \frac{r}{\mathscr{L}}$, $1 - \beta_0 = \frac{c}{\mathscr{L}}$, and $\zeta_L(s)$ admits two real zeros $\beta_0$ and $\beta_1$ such that  $\beta_0 \leq \beta_1$. It follows from \eqref{fL-expansion},  Lemmas \ref{Lemma2.4}- \ref{Lemma 2.6}, and \eqref{eqn:v_important} with $\varrho = \beta_0$ and $\varrho = \beta_1$
that
\begin{equation*}
\begin{split}
    f_L(\sigma, 0) 
    &\leq -\Big( \frac{1}{\sigma - \beta_1} + \frac{1}{\sigma - \beta_0}\Big) + \frac{1-\kappa}{2}\log d_{L} +  \frac{1}{\sigma -1} + \alpha_2 +  \Big( - \frac{(1-\kappa) \log \pi}{2} + d(0) \Big) n_L.
\end{split}
\end{equation*}
Therefore, $f_L(\sigma,0) \geq 0$ and $\mathscr{L} \geq \mathscr{L}_0 > 0$ imply
\begin{align}
    0 
    &\leq - \Big( \frac{1}{\sigma - \beta_1} + \frac{1}{\sigma - \beta_0}\Big) + \mathscr{L} \Big(\frac{1-\kappa}{2} + \frac{1}{r}\Big) + \alpha_2 +  \Big( d(0) - \frac{(1-\kappa) \log \pi}{2} \Big) n_L \nonumber\\
    &= \mathscr{L} \Big(\frac{1}{r} - \frac{1}{r+\nu} -  \frac{1}{r+c} + \frac{1 - \kappa}{2}\Big) + \alpha_2 +  \Big( d(0) - \frac{(1-\kappa) \log \pi}{2} \Big) n_L \nonumber\\
    &\leq \mathscr{L} \Big(\frac{1-\kappa}{2} + \frac{1}{r} - \frac{1}{r+\nu} -  \frac{1}{r+c} + \frac{\mathcal{U}(D_{n_L}) D_{n_L}}{\mathscr{L}_0}\Big) , \label{eqn:ub_important}
\end{align}
where $\mathcal{U}(x)$ and  $D_{n_0}$ are defined as in \eqref{unit-step-func} and \eqref{D0case5}, respectively.
Rearranging \eqref{eqn:ub_important} then yields
\begin{equation}\label{eqn:ub_important1-exception}
    c \geq \widehat{\eta}(n_L,r,\nu) := \frac{1}{\frac{1-\kappa}{2} + \frac{1}{r} - \frac{1}{r + \nu} + \frac{\mathcal{U}(D_{n_L}) D_{n_L}}{\mathscr{L}_0}} - r .
\end{equation}
Finally, it remains to compute $\hat{\eta}(n_L,r,\nu)$. For $\nu = 0.5$ and $0.05$, we complete the computations with the same optimisation strategy mentioned beneath  \eqref{eqn:ub_important1}; the results are given in Table \ref{tab:optimisations_ez_nu=0.5&0.05}.

\begin{table}[]
    \centering
    \begin{tabular}{ccccc}
        $n_L$ & $r_1$ & $\hat{\eta}(n_L,r_1,0.5)^{-1}$ & $r_2$ & $\hat{\eta}(n_L,r_2,0.05)^{-1}$ \\
        \hline
        2 & 1.48920 & 1.32086 & 0.949128 & 0.478632 \\
        3 & 0.822090 & 1.86631 & 0.822090 & 0.483802 \\
        4 & 0.865709 & 1.77210 & 0.865709 & 0.480747 \\
        5 & 1.10750 & 1.45550 & 0.949128 & 0.478632 \\
        6 & 1.37770 & 1.33079 & 0.949128 & 0.478632 \\
        $\geq$ 7 & 1.48920 & 1.32086 & 0.949128 & 0.478632 \\
    \end{tabular}
    \caption{Optimised parameter choices and computations for $\widehat{\eta}(n_L,r,0.5)$ and $\widehat{\eta}(n_L,r,0.05)$.}
    \label{tab:optimisations_ez_nu=0.5&0.05}
\end{table}

\newpage

\bibliographystyle{amsplain}
\bibliography{refs}

\newpage

\appendix

\section{Tables}\label{app:tables}

Table \ref{tab:d(k)} provides the admissible values of $\mathfrak{d}(\delta,k)$ for Lemma \ref{lem:McCurley_L1_refined}. Tables \ref{tab:case1_results_16}, \ref{tab:case1_results_40}, and \ref{tab:case1_results_46} contain computations toward Theorem \ref{thm:large_ordinates}.

\begin{table}[H]
\centering
\begin{small}
\begin{tabular}{c|cc}
    $k$  & $\delta=0$ & $\delta= 1$ \\
    \hline
    1 & -0.199351128738030570 & -0.0813946816693186803 \\
    2 & 
    -0.00124241978210990096 & 0.0273794535972986669 \\
       3 & 
       0.114084432427194876 & 0.122300543665333397 \\
       4 & 0.193622937944706114 & 0.196875436840859980\\
       5 & 0.254905404299203220 & 0.256570417484008384 \\
       6 & 0.304939896564586199 & 0.305953879673129092 \\
       7 & 0.347272949065725700 & 0.347963044467412053 \\
       8 & 0.383977593776409942 & 0.384482381069023826 \\
       9 & 0.416381018169095007 & 0.416768771337744870 \\
       10 & 0.445387629210662161 & 0.445696071325147825\\
       11 & 0.471642652575186894 & 0.471894504563624251 \\
       12 & 0.495623005500675062 & 0.495832877025740915 \\
       13 & 0.517691340292122715 & 0.517869110940852861 \\
       14 & 0.538129924980734864 & 0.538282546011518703 \\
       15 & 0.557162798876846832 & 0.557295319763611574 \\ 
       16 & 0.574970772272345387 & 0.575086958395945369 \\
       17 & 0.591701881908478833 & 0.591804603327346634 \\
       18 & 0.607478862713224155 & 0.607570348191724730 \\
       19 & 0.622404603717508165 & 0.622486612036975084 \\
       20 & 0.636566207972667386 & 0.636640146614778746 \\
       21 & 0.650038064524389947 & 0.650105073828502023 \\
       22 & 0.662884207683184790 & 0.662945221743776680 \\
       23 & 0.675160153278594466 & 0.675215944805046542 \\
       24 & 0.686914345156188544 & 0.686965559046277541 \\
       25 & 0.698189307163083250 & 0.698236485988400823 \\
       26 & 0.709022569768260946 & 0.709066173368575803 \\
       27 & 0.719447422233358336 & 0.719487842956925361 \\
       28 & 0.729493528314488215 & 0.729531102997373782 \\ 
       29 & 0.739187434165484158 & 0.739222453638918475 \\
       30 & 0.748552990323224754 & 0.748585707026276914 \\
       31 & 0.757611704643699757 & 0.757642338768938384 \\
       32 & 0.766383039316073544 & 0.766411783808659042 \\
       33 & 0.774884662259901380 & 0.774911686912753805 \\
       34 & 0.783132661061269442 & 0.783158115891840700 \\
       35 & 0.79114172595177534 & 0.791165744003572424 \\
       36 & 0.798925307054014988 & 0.798948006734258587 \\
       37 & 0.806495750117107169 & 0.806517237157871714 \\
       38 & 0.813864414178669282 & 0.813884783290518610 \\
       39 & 0.821041773965830557 & 0.821061110238844649 \\
       40 & 0.828037509350176948 & 0.828055889446331683 \\
       41 & 0.834860583772007780 & 0.834878076944323810 \\
       42 & 0.841519313226652188 & 0.841535982193814758 \\
       43 & 0.848021427143618833 & 0.848037328843449179 \\
       44 & 0.854374122275479042 & 0.854389308516414592 \\
       45 & 0.860584110537916169 & 0.860598628564266499 \\
       46 & 0.866657661597714180 & 0.866671554581723291
   \end{tabular}
   \end{small}
\caption{Values of $\mathfrak{d}(\delta,k)$}
\label{tab:d(k)}
\end{table}

\begin{table}[H]
    \centering
    \begin{small}
    \begin{tabular}{c|ccccc}
        $n_0$ & $\varepsilon$ & $C_1$ & $C_2$ & $C_3$ & $C_4$ \\
        \hline
        $3$ & 0.1295 & 12.24107 & 9.53466 & -11.79351 & 4.7255 \\
        $4$ & 0.0735 & 12.24107 & 9.53466 & -12.06914 & 3.60145 \\
        $5$ & 0.0311 & 12.24107 & 9.53466 & -12.18761 & 2.7179 \\
        $6$ & 0.0240 & 12.24107 & 9.53466 & -12.20789 & 2.56754 \\
        $7$ & 0.0154 & 12.24107 & 9.53466 & -12.23261 & 2.38456 \\
        $8$ & 0.0131 & 12.24107 & 9.53466 & -12.23926 & 2.33546 \\
        $9$ & 0.0101 & 12.24107 & 9.53466 & -12.24795 & 2.27133 \\
        $10$ & 0.0091 & 12.24107 & 9.53466 & -12.25085 & 2.24992 \\
        $11$ & 0.0073 & 12.24107 & 9.53466 & -12.25608 & 2.21137 \\
        $12$ & 0.0067 & 12.24107 & 9.53466 & -12.25782 & 2.19851 \\
        $13$ & 0.0057 & 12.24107 & 9.53466 & -12.26073 & 2.17707 \\
        $14$ & 0.0053 & 12.24107 & 9.53466 & -12.26190 & 2.16848 \\
        $15$ & 0.0046 & 12.24107 & 9.53466 & -12.26394 & 2.15346 \\
        $16$ & 0.0044 & 12.24107 & 9.53466 & -12.26452 & 2.14917 \\
        $17$ & 0.0039 & 12.24107 & 9.53466 & -12.26598 & 2.13844 \\
        $18$ & 0.0037 & 12.24107 & 9.53466 & -12.26656 & 2.13414 \\
        $19$ & 0.0033 & 12.24107 & 9.53466 & -12.26773 & 2.12555 \\
        $20$ & 0.0032 & 12.24107 & 9.53466 & -12.26802 & 2.1234 \\
        $21$ & 0.0030 & 12.24107 & 9.53466 & -12.2686 & 2.11911
    \end{tabular}
    \end{small}
    \caption{Admissible computations for $C_1$, $C_2$, $C_3$, $C_4$ upon choosing the polynomial $p_{16}$ such that \eqref{eqn:dlvp_nf} holds with $T=1$.}
    \label{tab:case1_results_16}
\end{table}

\begin{table}[H]
    \centering
    \begin{small}
    \begin{tabular}{c|ccccc}
        $n_0$ & $\varepsilon$ & $C_1$ & $C_2$ & $C_3$ & $C_4$ \\
        \hline
        $3$ & 0.1250 & 12.21608 & 9.53979 & -11.63414 & 4.60646 \\
        $4$ & 0.0708 & 12.21608 & 9.53979 & -11.89009 & 3.52232 \\
        $5$ & 0.0305 & 12.21608 & 9.53979 & -12.0016 & 2.68709 \\
        $6$ & 0.0235 & 12.21608 & 9.53979 & -12.02138 & 2.53978 \\
        $7$ & 0.0152 & 12.21608 & 9.53979 & -12.04498 & 2.3643 \\
        $8$ & 0.0130 & 12.21608 & 9.53979 & -12.05127 & 2.31765 \\
        $9$ & 0.0100 & 12.21608 & 9.53979 & -12.05986 & 2.25394 \\
        $10$ & 0.0090 & 12.21608 & 9.53979 & -12.06273 & 2.23267 \\
        $11$ & 0.0072 & 12.21608 & 9.53979 & -12.0679 & 2.19437 \\
        $12$ & 0.0067 & 12.21608 & 9.53979 & -12.06934 & 2.18373 \\
        $13$ & 0.0056 & 12.21608 & 9.53979 & -12.07251 & 2.16029 \\
        $14$ & 0.0053 & 12.21608 & 9.53979 & -12.07337 & 2.1539 \\
        $15$ & 0.0046 & 12.21608 & 9.53979 & -12.07539 & 2.13898 \\
        $16$ & 0.0043 & 12.21608 & 9.53979 & -12.07625 & 2.13258 \\
        $17$ & 0.0039 & 12.21608 & 9.53979 & -12.07741 & 2.12405 \\
        $18$ & 0.0037 & 12.21608 & 9.53979 & -12.07798 & 2.11979 \\
        $19$ & 0.0033 & 12.21608 & 9.53979 & -12.07914 & 2.11125 \\
        $20$ & 0.0032 & 12.21608 & 9.53979 & -12.07943 & 2.10912 \\
        $21$ & 0.0030 & 12.21608 & 9.53979 & -12.0800 & 2.10485
    \end{tabular}
    \end{small}
    \caption{Admissible computations for $C_1$, $C_2$, $C_3$, $C_4$ upon choosing the polynomial $p_{40}$ such that \eqref{eqn:dlvp_nf} holds with $T=1$.}
    \label{tab:case1_results_40}
\end{table}

\begin{table}[H]
    \centering
    \begin{small}
    \begin{tabular}{c|ccccc}
        $n_0$ & $\varepsilon$ & $C_1$ & $C_2$ & $C_3$ & $C_4$ \\
        \hline
        $3$ & 0.1239 & 12.21124 & 9.54177 & -11.59548 & 4.57803 \\
        $4$ & 0.0701 & 12.21124 & 9.54177 & -11.84681 & 3.50267 \\
        $5$ & 0.0303 & 12.21124 & 9.54177 & -11.9567 & 2.67879 \\
        $6$ & 0.0234 & 12.21124 & 9.54177 & -11.97615 & 2.53379 \\
        $7$ & 0.0151 & 12.21124 & 9.54177 & -11.9997 & 2.35857 \\
        $8$ & 0.0129 & 12.21124 & 9.54177 & -12.00597 & 2.31198 \\
        $9$ & 0.0099 & 12.21124 & 9.54177 & -12.01454 & 2.24836 \\
        $10$ & 0.0089 & 12.21124 & 9.54177 & -12.0174 & 2.22713 \\
        $11$ & 0.0072 & 12.21124 & 9.54177 & -12.02227 & 2.19101 \\
        $12$ & 0.0067 & 12.21124 & 9.54177 & -12.02371 & 2.18038 \\
        $13$ & 0.0056 & 12.21124 & 9.54177 & -12.02686 & 2.15698 \\
        $14$ & 0.0053 & 12.21124 & 9.54177 & -12.02773 & 2.1506 \\
        $15$ & 0.0046 & 12.21124 & 9.54177 & -12.02974 & 2.1357 \\
        $16$ & 0.0043 & 12.21124 & 9.54177 & -12.0306 & 2.12931 \\
        $17$ & 0.0038 & 12.21124 & 9.54177 & -12.03204 & 2.11866 \\
        $18$ & 0.0037 & 12.21124 & 9.54177 & -12.03233 & 2.11653 \\
        $19$ & 0.0033 & 12.21124 & 9.54177 & -12.03348 & 2.10801 \\
        $20$ & 0.0032 & 12.21124 & 9.54177 & -12.03377 & 2.10588 \\
        $21$ & 0.0030 & 12.21124 & 9.54177 & -12.03434 & 2.10162
    \end{tabular}
    \end{small}
    \caption{Admissible computations for $C_1$, $C_2$, $C_3$, $C_4$ upon choosing the polynomial $p_{46}$ such that \eqref{eqn:dlvp_nf} holds with $T=1$.}
    \label{tab:case1_results_46}
\end{table}


\begin{table}[H]
\centering
\begin{small}
\begin{tabular}{cl|cl}
    $k$ & $a_k$ & $k$ & $a_k$ \\
    \hline 
    0 & 1 & 21  & $4.66702819061453\cdot 10^{-7}$ \\ 
    1 & 1.74600190914994        & 22 & $8.88183754657211\cdot 10^{-7}$ \\ 
    2 & 1.14055431833244        & 23 & $6.61799442215331\cdot 10^{-5}$ \\ 
    3 & 0.518966962914028       & 24 & $3.70153227317542\cdot 10^{-5}$ \\ 
    4 & 0.130885859164882       & 25 & $6.2332255794641\cdot 10^{-8}$  \\
    5 & $8.86418531143308\cdot 10^{-8}$ & 26 & $3.29243016002061\cdot 10^{-5}$ \\ 
    6 & $1.79787121328335\cdot 10^{-6}$ & 27 & $4.89938220699415\cdot 10^{-5}$ \\ 
    7 & 0.0137716529944408      & 28 & $1.50988491954013\cdot 10^{-5}$ \\ 
    8 & 0.00825900683475376     & 29 & $1.13051732969427\cdot 10^{-7}$ \\ 
    9 & $4.91544374578637\cdot 10^{-6}$ & 30 & $2.11823533257304\cdot 10^{-5}$ \\ 
    10  & $2.20263007866541\cdot 10^{-6}$ & 31 & $2.13859401551174\cdot 10^{-5}$ \\ 
    11 & 0.00243120523137902     & 32 & $1.55071932288034\cdot 10^{-6}$ \\ 
    12 & 0.00172926530269636     & 33 & $1.51812185041036\cdot 10^{-6}$ \\ 
    13 & $1.35500078722447\cdot 10^{-6}$ & 34 & $1.67615806595912\cdot 10^{-5}$ \\ 
    14 & $2.20879127662495\cdot 10^{-6}$ & 35 & $1.60031224178442\cdot 10^{-5}$ \\ 
    15 & 0.00069712400164774     & 36 & $3.94634065729451\cdot 10^{-6}$ \\ 
    16 & 0.000530583559753362    & 37 & $4.08859029078879\cdot 10^{-7}$ \\ 
    17 & $6.3973072524226\cdot 10^{-7}$  & 38 & $1.77819241241605\cdot 10^{-6}$ \\ 
    18 & $5.37323136636712\cdot 10^{-7}$ & 39 & $5.06885733758335\cdot 10^{-8}$ \\ 
    19 & 0.000234320877800568 & 40 & $7.50406436813653\cdot 10^{-9}$ \\
    20  & 0.000177364641910045 &   &  
\end{tabular}
\end{small}
\caption{Table of coefficients for Mossinghoff--Trudgian--Yang's polynomial $p_{40}(\varphi)\in P_{40}$.}
\label{table:coefficients2}
\end{table}

\begin{table}[H]
\centering
\begin{small}
\begin{tabular}{cl|cl}
    $k$ & $a_k$ & $k$ & $a_k$ \\
    \hline 
    0  & 1                        & 24 & 0.000127104592072581     \\
    1  & 1.74708744081848         & 25 & $1.74058423843506\cdot 10^{-7}$ \\
    2  & 1.14338015090023         & 26 & $6.156980223188\cdot 10^{-9  }$ \\
    3  & 0.521864216745001        & 27 & $7.4923012998548\cdot 10^{-5 }$ \\
    4  & 0.132187571762225        & 28 & $6.29610657045172\cdot 10^{-5}$ \\
    5  & $1.44250682908725\cdot 10^{-7}$  & 29 & $4.51492091998615\cdot 10^{-7}$ \\
    6  & $4.69075278525482\cdot 10^{-9}$  & 30 & $1.76696516341167\cdot 10^{-8}$ \\
    7  & 0.0141904926848435       & 31 & $3.57616762286565\cdot 10^{-5}$ \\
    8  & 0.00859097729886965      & 32 & $2.9356535048273\cdot 10^{-5 }$ \\
    9  & $5.05758761820625\cdot 10^{-7}$  & 33 & $2.6547976338407\cdot 10^{-7 }$ \\
    10 & $4.42284301054098\cdot 10^{-10}$ & 34 & $7.39578841754684\cdot 10^{-7}$ \\
    11 & 0.00262452919575262      & 35 & $1.5703528751761\cdot 10^{-5 }$ \\
    12 & 0.0018969952017721       & 36 & $1.16349907747152\cdot 10^{-5}$ \\
    13 & $4.69472495111911\cdot 10^{-10}$ & 37 & $1.01423339047177\cdot 10^{-7}$ \\
    14 & $2.18058618368512\cdot 10^{-7}$  & 38 & $1.71248131672039\cdot 10^{-6}$ \\
    15 & 0.000818384876659817     & 39 & $7.84636117271159\cdot 10^{-6}$ \\
    16 & 0.000639651965532567     & 40 & $5.93829512034697\cdot 10^{-6}$ \\
    17 & $3.11262094946825\cdot 10^{-8}$  & 41 & $9.47232309558493\cdot 10^{-7}$ \\
    18 & $7.74994211145798\cdot 10^{-7}$  & 42 & $4.84440446543232\cdot 10^{-8}$ \\
    19 & 0.000329183630974004     & 43 & $9.72548049252508\cdot 10^{-7}$ \\
    20 & 0.000268358318561904     & 44 & $8.45180184576162\cdot 10^{-7}$ \\
    21 & $4.43747297378809\cdot 10^{-7}$  & 45 & $2.25111200007826\cdot 10^{-7}$ \\
    22 & $1.87358718910571\cdot 10^{-7}$  & 46 & $6.56678999833493\cdot 10^{-10}$ \\
    23 & 0.000151428354073652     &    &    
\end{tabular}
\end{small}
\caption{Table of coefficients for Mossinghoff--Trudgian--Yang's polynomial $p_{46}(\varphi)\in P_{46}$.}
\label{table:coefficients3}
\end{table}

\end{document}